\documentclass[11pt]{amsart}
\usepackage{amsthm}
\newtheorem{theorem}{Theorem}[section]
\newtheorem{corollary}[theorem]{Corollary}
\newtheorem{lemma}[theorem]{Lemma}
\newtheorem{proposition}[theorem]{Proposition}
\newtheorem{definition}[theorem]{Definition}
\makeatletter
\def\th@newremark{\th@remark\thm@headfont{\bfseries}}
\makeatletter
\usepackage{enumitem}
\usepackage{fancyhdr}
\theoremstyle{newremark}
\newtheorem{remark}{Remark}
\usepackage{graphicx} 
\usepackage{amsmath,color,mathtools}
\usepackage{amssymb, bbm}
\usepackage[margin=0.8in]{geometry}
\usepackage[pdfencoding=auto,unicode,psdextra]{hyperref}

\title{Monotonicity of functionals associated to product measures via their Fourier transform and applications}
\author{Andreas Malliaris}
\address{Institut de Mathématiques de Toulouse (UMR 5219). Université de Toulouse
\& CNRS. UPS, F-31062 Toulouse Cedex 09, France}
\email{andreas.malliaris@math.univ-toulouse.fr}
\date{}
\begin{document}

\pagestyle{plain}

\begin{abstract}
    Let $\mu$ be a probability measure on $\mathbb{R}$. We give conditions on the Fourier transform of its density for functionals of the form $H(a)=\int_{\mathbb{R}^n}h(\langle a,x\rangle)\mu^n(dx)$ to be Schur monotone. As applications, we put certain known and new results under the same umbrella, given by a condition on the Fourier transform of the density. These results include certain moment comparisons for independent and identically distributed random vectors, when the norm is given by intersection bodies, and the corresponding vector Khinchin inequalities.  We also extend the discussion to higher dimensions.
\end{abstract}

\maketitle

\section{Introduction}

\subsection{Motivation} The motivation for the present note comes from two classical extremal problems, that in recent years have attracted a lot of attention. The first one deals with extremal sections of convex bodies and the second one is about comparison of moments of sums of independent random variables.

The first problem is as follows: given $K$ a convex body in $\mathbb{R}^n$, what are the maximal and minimal values of the $(n-1)$-dimensional volume of central sections $|K\cap \theta^\perp|$, when $\theta\in S^{n-1}$? Here  we denote $\theta^\perp:=\{x\in \mathbb{R}^n  :  \langle x,\theta\rangle=0\}$.

A significant focus was given to the case when $K$ is the unit ball of the space $l_p^n$, which we denote by $B_p^n:= \{x\in \mathbb{R}^n : |x_1|^p+\cdots+|x_n|^p\leq 1\}$, for $p\in [1,\infty)$ and $B_\infty^n=\{x\in \mathbb{R}^n : \max |x_i|\leq 1\}$. A mixture of geometric, analytic and probabilistic techniques were used for determining the behavior of $|B_p^n\cap \theta^\perp|$. We refer to the classical works \cite{ball}, \cite{hadwiger} and \cite{hensley} for $p=\infty$, to \cite{koldobsky98}, \cite{MP} and to the more recent \cite{eskenazis} for unifying and generalizing several results in the ``higher" dimensional setting and to the recent survey \cite{NT} for a more complete picture on sections and projections of $p$-balls. The problem is solved for $p\in (0,2]$, however the maximal section of $B_p^n$ was found for $p$ large enough by Eskenazis, Nayar and Tkocz only very recently \cite{ENTDuke}. The above questions prompt us to consider the following scenario:

Let $\mu$ be a probability measure on $\mathbb{R}^n$, for any Borel set $A \subset \mathbb{R}^n$ define its $\mu$-boundary measure (or lower Minkowski content) via 
\[
\mu^+(A) = \liminf_{\varepsilon \to 0^+} \frac{
\mu(A + \varepsilon B_2^n) - \mu(A)}{\varepsilon}\cdot
\]
Then if we denote $\theta^\perp_+=\{x\in \mathbb{R}^n  :  \langle x,\theta\rangle\geq0\}$, it is natural to ask the following question, which is the starting point of the investigation in the present note: \par
\textit{Question.} For which direction $\theta\in S^{n-1}$ is the maximum and for which the minimum of
\begin{equation}\label{centralmeas}
\mu^+(\theta^\perp_+)
\end{equation}
achieved?

The answer to this question in complete generality would solve of course all the geometric problems around sections of convex bodies. The Fourier analytic approach of Ball for the case $p=\infty$, along  with the independence (product structure) of the cube, suggest to restrict the question to the setting of product probability measures. Even though for the case of $p$-balls the product structure seems irrelevant, Meyer and Pajor in \cite{MP} stated it equivalently as a problem for a product measure, that was used by Koldobsky to solve the problem for $p\in (0,2]$.
Koldobsky proved even more, i.e. that the quantity $|B_p^n\cap\theta^\perp|$ is Schur convex in $(\theta_1^2,\ldots,\theta_n^2)$, a result further generalized by Eskenazis, see Theorem \ref{eskethm}. The problem of sections of product measures was considered also by Barthe in \cite{bartheextremal}, which he approached using the notion of peakedness and also for subspaces of higher co-dimension.

The second problem is stated in probabilistic terms: given $X_1,\ldots,X_n$ i.i.d. random variables, which directions $\theta\in S^{n-1}$ extremize the $p$-norm of
\[
\|\theta_1X_1+\cdots+\theta_nX_n\|_p:= \left(\mathbb{E} |\theta_1X_1+\cdots+\theta_nX_n|^p \right)^{\frac{1}{p}} \ ?
\]

When the independent random variables are uniformly distributed on $[-1,1]$, or Gaussian mixtures, these problems were studied in \cite{latole} and \cite{ENT} respectively. In these papers, the authors proved something more, namely that the quantity
\[
\mathbb{E}|\theta_1X_1+\cdots+\theta_nX_n|^p
\]
is Schur monotone in $(\theta_1^2,\ldots,\theta_n^2)$, depending on $p$. 

Moreover, functions of the form 
\[
\mathbb{E}\Phi(\theta_1U_1+\cdots+\theta_nU_n)
\]
were considered in \cite{culverhouse}, for $U_i$ independent random vectors uniformly distributed on the unit sphere and $\Phi$ being bi-subharmonic, where again a Schur monotonicity result was proven.

The moment comparison question, as we shall see below, is also related with finding sharp constants in Khinchin inequalities, of the form 
\[
\|S\|_p\leq C_{p,q}\|S\|_q,
\]
 where $S$ is a normalized sum of random variables as above. For more references we refer to the above articles.

\subsection{Main results and organization}
Inspired by the above questions it seems natural to consider more general functionals of the form \[\int_{\mathbb{R}^n}h(\langle\theta,x \rangle)\mu^n(dx),\] for $\theta \in S^{n-1}$ and suitable functions $h$, and study the question on finding the maximizer and minimizer of such functional, depending on the probability measure involved.
The main observation - that we state here for measures in one dimension - of the present note is the following, and can be applied to both problems stated in the introduction: 
\begin{theorem} \label{mainthmintro}
    Let $\mu$ be a probability measure on $\mathbb{R}$ with an even density $\phi$ which has positive and integrable Fourier transform and $h:\mathbb{R}\to \mathbb{R}$ be an even function, in $L^2$ and $\Hat{h}$ positive. Let  
    \[
    H(a)=H(a_1,\ldots,a_n)=\int_{\mathbb{R}^{n}}h\left( a_1^{\frac{1}{q}}x_1+\cdots+a_n^{\frac{1}{q}}x_n\right)\mu^n(dx) , \:\:\: a\in [0,+\infty)^n.
    \]
$\bullet$ If $\Hat{\phi}(t^{\frac{1}{q}})$, for $t\geq0$ is log-convex, then $H$ is log-convex.

\noindent$\bullet$ If $\Hat{\phi}(t^{\frac{1}{q}})$, for $t\geq0$ is log-concave, then $H$ is Schur concave.
    
\end{theorem}

In Section \ref{Section main} we prove Theorem \ref{mainthmintro} in a more general form and in fact a converse statement as well, along with the analogue for sections. As a corollary we also get the following result which complements and extends in some cases Theorem 3 of \cite{ENT}.

\begin{corollary} \label{corollary intro one dim}
     Let $q>0$ and $X_1,\ldots,X_n$ be i.i.d. random variables distributed according to $\mu(dx)=\phi(x)dx$, where $\phi$ is even with integrable and positive Fourier transform $\hat{\phi}$ satisfying $\Hat{\phi}(t^\frac{1}{q})$ is log-convex on $(0,+\infty)$.
    
Then for each $\lambda>0$, the function, defined for $a\in (0,+\infty)^n$,
\begin{equation} \label{boosted 1 dim}
 \mathbb{E} e^{-\lambda|\sqrt[q]{a_1}X_1+\cdots+ \sqrt[q]{a_1}X_n|^p}
\end{equation}
is log-convex.
    
    In particular, assuming that   $\mathbb{E} |X_1|^p<+\infty$, for some $p\in (0,2]$ the function 
\[
\mathbb{E} |a_1^{\frac{1}{q}}X_1+\cdots+ a_n^{\frac{1}{q}}X_n|^p
\]
is concave in $a$.

While, if  $\Hat{\phi}(t^\frac{1}{q})$ is log-concave on $(0,+\infty)$ one has Schur concavity and Schur convexity respectively.
\end{corollary}

In \cite{ENT}, the result for the $p$ moments is proven for Gaussian mixtures and $q=2$, for the whole range of exponents. The statement (\ref{boosted 1 dim}) is formally stronger, since it allows infinite moments, besides the larger class of densities. The results for the whole range of exponents $p$ and $q$ (especially when $q=2$) seem to be new in the case when the log-concavity assumption is satisfied.

An important application of such result is to obtain sharp Khinchin inequalities, as indicated previously, that one can get for all the measures satisfying the condition that $t\mapsto \hat{\phi}(\sqrt{t})$ is either log-concave or log-convex and is the content of Corollary \ref{corollary khinchine log-concave}. Again, apart from the Gaussian mixtures in the log-convex case, the rest seem new. Some further applications are discussed in Section \ref{Section main}, that mainly concern the one dimensional case, but also examples of densities satisfying the assumptions for each $q$ in one and higher dimensions.

In section \ref{Section High dim} we investigate some applications of a higher dimensional analogue of Theorem \ref{mainthmintro}.
The first part of it is mainly concerned with the $p$-balls generated from a convex body $K\subset \mathbb{R}^d$, defined as 
\[B_p^n(K):=\{x=(x_1,\ldots,x_n)\in \mathbb{R}^{nd} : \|x_1\|^p_K+\cdots+\|x_n\|_K^p\leq1\}.\]
Following the work of Barthe, Guédon, Mendelson and Naor \cite{BGMN}, we give a representation of the uniform probability measure on $B_p^n(K)$ which might be of independent interest. Then we give some applications of a higher dimensional analogue of Theorem \ref{mainthmintro}, when $K$ is the unit ball of a finite dimensional subspace of $L_p$, to (block) sections of $B_p^n(K)$, which is Theorem \ref{eskethm}, and the main result of \cite{eskenazis}. The advantage here is a somewhat quicker argument by essentially showing that the method of Meyer-Pajor and Koldobsky works in that realm as well with a slightly stronger statement, while it is observed that the argument given in \cite{eskenazis} via comparison of Laplace transforms is again equivalent to a problem of sections of product measures, see Remark \ref{remark laplace is section}. Moreover we show moment comparison of random vectors (Corollary \ref{highmom}) that is the higher dimensional analogue of Corollary \ref{corollary intro one dim}. This, in conjunction with the representation of the uniform probability measure on $B_p^n(K)$ gives moment comparisons for a random vector uniformly distributed in $B_p^n(K)$ and is the content of Corollary \ref{corollary moments random vector}.

Moreover, in section \ref{Section High dim} we prove comparison (and actually log-convexity) results using norms induced by intersection bodies. These were introduced by Lutwak in \cite{Lutwak}, and are defined as follows: a symmetric star body $K\subset \mathbb{R}^d$ is called the intersection body of a body $L\subset \mathbb{R}^d$, if for each direction $\theta\in S^{d-1}$ it holds that $\|\theta\|_K=|L\cap\theta|^{-1}$. In particular one has the following:

\begin{theorem} \label{intersectionintro}
     Let $X_1,\ldots,X_n$ be independent random vectors in $\mathbb{R}^d$ distributed according to the probability measure $\mu(dx)=\phi(x)dx$. Suppose that $\phi$ is even and has positive integrable Fourier transform and be such that for each $t\in \mathbb{R}^d$ the function $[0,+\infty)\ni r\mapsto\Hat{\phi}(\sqrt{r}t)$ is log-convex. Assume moreover that for each $N>0$, $\mathbb{E}_\mu|X_1|^N<\infty$. Let also $K$ be a symmetric star body on $\mathbb{R}^d$ that is the intersection body of $L$.
    
    Then, if $\mathbb{E}\|X_1\|_K^{-1}<\infty$, the function
\[
\mathbb{E} \|\sqrt{a_1}X_1+\cdots+\sqrt{a_n}X_n\|_K^{-1}
\]
is log-convex in $(a_1,\ldots,a_n)\in [0,+\infty)^n$.
\end{theorem}

The crucial property here is that $\|\cdot\|_K^{-1}$ represents a positive definite distribution.
This in general will be true when $\|\cdot\|_K^{-1}$ is replaced by $\|\cdot\|^{-l}$ and $(\mathbb{R}^d,\|\cdot\|)$ is a quasi-normed space that embeds in $L_{-l}$.
For the most general results see Theorem \ref{momentsweak}.
Such moment comparison can give sharp vector Khinchine inequalities for negative moments, once the convex body is an ellipsoid. This is the content of Corollary \ref{khinchin}. Khinchine inequalities with negative moments were proved also in \cite{chasap} in connection to Ball's integral inequality.

The final section contains comments related to Theorem \ref{mainthmintro}, on higher co-dimensional sections and on mixtures of random variables highlighting the difference in the arguments compared to the Fourier approach emphasized in the rest note.

 All the relevant notions and certain auxiliary results are introduced in the following Section \ref{section prelims}.

\section{Preliminary notions} \label{section prelims}
If $C\subset \mathbb{R}^d$ is convex, then a function $f:C\to [0,+\infty)$ is called log-convex (resp. log-concave) if $\log(f):C\to\mathbb{R}\cup \{-\infty\}$ is a convex (resp. concave) function. The Euclidean norm of $v\in \mathbb{R}^d$ is denoted by $|v|$ and if $A$ is a Borel subset of $\mathbb{R}^d$, by $|A|$ we mean its Lebesgue measure. The group of permutations of $[n]=\{1,\ldots,n\}$ is denoted by $S_n$.

Our convention for the Fourier transform is that 
\[
\hat{f}(x) = \int_{\mathbb{R}^d}f(y)e^{-2\pi i \langle x,y\rangle}dy,
\]
while if $\nu$ is a finite Borel measure on $\mathbb{R}^d$ we denote $\hat{\nu}$ its Fourier transform
\[
\hat{\nu}(x)= \int_{\mathbb{R}^d}e^{-2\pi i \langle x,y\rangle}\nu(dy).
\]

Below we make use of the coarea formula (see e.g. \cite[Theorem 2.86]{giaquintabook}): If $l\leq k$ and $F:\mathbb{R}^k\to \mathbb{R}^l$ is $C^1$, then for all measurable functions $u:\mathbb{R}^k\to \mathbb{R}$ it holds that
\[
\int_{\mathbb{R}^k} u(x)J_F(x)dx= \int_{\mathbb{R}^l}\int_{\{F=z\}}u(x)\mathcal{H}^{k-l}(dx)dz,
\]
provided the expression on the left hand-side is integrable, where $J_F(x):=\sqrt{\det\left(DF(x)DF(x)^t\right)}$ is the Jacobian of $F$ at $x$ and $\mathcal{H}^{k-l}$ is the $(k-l)$-dimensional Hausdorff measure.

\subsection{Schur monotonicity} \label{Schur background}
 We now recall the notion of majorization and Schur order. See e.g. the book \cite{Marshall-book}.

 Let $x,y\in \mathbb{R}^n$. We say that $y$ majorizes $x$ and we write $x \prec y$ if, after rearranging the terms so that $x_1\geq \cdots \geq x_n$  and $y_1\geq \cdots \geq y_n$, it holds that for all $i=1,\ldots,n-1$ $x_1+\cdots+x_i\leq y_1+\cdots+y_i$ and $x_1+\cdots+x_n=y_1+\cdots+y_n$.

 Recall that $P$ is a doubly stochastic matrix if it has non-negative entries and the entries of each row and of each column sum up to 1.
 A useful observation of Hardy and Littlewood (Theorem 2.B.2 in \cite{Marshall-book}) tells us that if $x \prec y$ then there exists a doubly stochastic matrix $P$ such that $x=Py$ (equivalently $x\in \mathrm{conv}\{(y_{\sigma_1},\ldots,y_{\sigma_n}), \sigma \in S_n\}$). The converse of the statement is obviously true as well.
 
  Moreover, we will use the fact that, if $x\in \mathbb{R}^n$ is such that $x_1+\cdots+x_n=1$ and $x_i\geq0$ for each $i$, then $(\frac{1}{n},\ldots,\frac{1}{n})\prec x \prec (1,0,\ldots,0)$.
 
A function will be called Schur monotone if it is monotone with respect to this order. In particular, a real-valued function $f \colon A \subset \mathbb{R}^n \to \mathbb{R} $ is Schur convex (resp. Schur concave), if for all $x,y\in A$ such that $x \prec y$ it holds $f(x) \leq f(y)$ (resp. $f(x) \geq f(y)$). 
The next statement characterizes Schur monotonicity
and is useful to verify that a function is Schur monotone:

 Let $A\subset \mathbb{R}^n$ be invariant under permutations of coordinates and let $f: A\to \mathbb{R}$ which has first partial derivatives. Then $f$ is Schur convex if and only if for all $x\in A$ and $i,j$
\begin{equation}\label{thm:characterization}
(x_i-x_j)(\partial_if(x)-\partial_jf(x))\geq 0 .
\end{equation}

Moreover it follows from the definitions that a convex function $f$, that is also symmetric in its arguments (i.e. $f(a_1,\ldots,a_n)=f(a_{\sigma(1)},\ldots,a_{\sigma(n)})$ for all $\sigma \in S_n)$ it is Schur convex, while the converse does not hold in general.

When dealing with random variables, the following result confirms the Schur monotonicity:
\begin{proposition}[\cite{Marshall-book} Section 11 Prop.B.2.] \label{exchangable theorem}
    If $X_1,\ldots,X_n$ are exchangeable random variables
and $f: \mathbb{R}^n\to \mathbb{R}$ is a symmetric convex function,
then the function $g$ defined by $g(a_1,\ldots,a_n) = \mathbb{E}f(a_1X_1,\ldots,a_nX_n)$ is symmetric and convex, and in particular it is Schur-convex.
\end{proposition}
Where we say that the random variables $X_1,\ldots,X_n$ are exchangeable if the vector $(X_{\sigma(1)},\ldots,X_{\sigma(n)})$ has distribution independent of $\sigma\in S_n$.

The following classical lemma, for which we provide the proof for completeness, will be useful for our purposes later on:

\begin{lemma} \label{Schurlemma}
    Let $f:[0,+\infty)\to [0,+\infty)$ be log-concave (resp. log-convex). Then for each $n\in \mathbb{N}$ the function $F: [0,+\infty)^n\to [0,+\infty)$ defined as \[F(a_1,\ldots,a_n)=\prod_{i=1}^n f(a_i)\] is Schur concave (resp. Schur convex).
    
    Moreover, if we assume that the function $f$ is continuous the converse holds.
\end{lemma}

\begin{proof}
    Let $a,b\in [0,+\infty)^n$ be such that $a \prec b$. By the Theorem of Hardy and Littlewood there exists a doubly stochastic matrix $P=(p_{ij})$ such that $a=Pb$, and so $a_i=\sum_j p_{ij}b_j$.

    Suppose that $f$ is log-concave, then
    \[
    F(a_1,...,a_n)=\prod^n_{i=1} f\left(\sum^n_{j=1} p_{ij}b_j\right)\geq \prod^n_{i=1}\prod^n_{j=1} f(b_j)^{p_{ij}} = \prod^n_{j=1} f(b_j)=F(b),
    \]
    since the matrix is doubly stochastic, and this shows the Schur concavity of the function. For the log-convexity one proceeds similarly.

    To see the converse, let $n=2$ and $a_1,a_2 > 0$. Then the Schur concavity of $F(x,y)=f(x)f(y)$ along with the fact that $\left( \frac{a_1+a_2}{2},\frac{a_1+a_2}{2}\right)\prec(a_1,a_2)$ gives 
    \[
    f\left( \frac{a_1+a_2}{2}\right)^2\geq f(a_1)f(a_2).
    \]
    Thus if $f$ is continuous we get that $f$ is indeed log-concave.
\end{proof}

\subsection{Positive definite functions and distributions} \label{pos def background}
Here we recall the definition of a positive definite function (see eg \cite{feller}). A function $h : \mathbb{R}^d\to \mathbb{R}$ (or more generally complex-valued) is called positive definite, if for each $k\in \mathbb{N}$ and every $x_1,\ldots,x_k\in \mathbb{R}^d$, the matrix $\left(h(x_i-x_j)\right)_{i,j}$ is positive semi-definite. Bochner's theorem ensures that a continuous function $h: \mathbb{R}^d\to \mathbb{R}$ is positive definite if and only if it is the Fourier transform of a finite (positive) Borel measure on $\mathbb{R}^d$, i.e. there exists a finite positive Borel measure on $\mathbb{R}^d$ such that
\[
h(x)=\int_{\mathbb{R}^d}e^{-2\pi i\langle x,y\rangle}d\nu(y).
\]

The Schwartz space $\mathcal{S}(\mathbb{R}^d)$ is defined to be the set of all (real valued is all we care here) $\psi\in C^\infty(\mathbb{R}^d)$, such that for all $\alpha=(\alpha_1,\ldots,\alpha_d)\in \mathbb{Z}^d_+$ and $N\in \mathbb{N}$, there exists $C_{N,\alpha}>0$ such that
\[
|\partial^\alpha\psi(x)|\leq C_{N,\alpha}\frac{1}{\left(  1+|x|   \right)^N}\cdot 
\]
We note that the following equivalence holds: $\psi\in \mathcal{S}(\mathbb{R}^d)\iff \Hat{\psi}\in \mathcal{S}(\mathbb{R}^d)$, and also the product of two Schwartz functions is again a Schwartz function.

Under the locally convex topology defined by the above sequence seminorms (see \cite{rudin} for more details), the continuous linear functionals on $\mathcal{S}(\mathbb{R}^d)$ are called tempered distributions. For example, a function $f$ that satisfies $\lim_{|x|\to \infty}f(x)/|x|^\gamma=0$, for some $\gamma>0$, gives rise to a tempered distribution $T_f$ via the usual action 
\[
T_f(\psi)=\int_{\mathbb{R}^d}f(x)\psi(x)dx.
\]

By Parseval's formula, one can define the Fourier transform of a tempered distribution $T$, to be the tempered distribution $\Hat{T}$, defined via
\[
\Hat{T}(\psi)=T(\Hat{\psi}) , \text{   for   } \psi \in \mathcal{S}(\mathbb{R}^d).
\]
Finally, a distribution $T$ will be called positive definite if for all non-negative $\psi \in \mathcal{S}(\mathbb{R}^d)$ it holds $\Hat{T}(\psi)\geq 0$.

In that setting Schwartz generalized Bochner's theorem by proving that a positive definite distribution is the Fourier transform of a tempered measure.

\subsection{Convex bodies} \label{convex bodies background}

A set $K\subset \mathbb{R}^d$ is called a convex body if it is convex, compact with non-empty interior and is called a symmetric convex body if in addition $K=-K$. We will say that $K$ is a star body if it is compact and for each $x\in K$ and $t\in [0,1]$, $tx\in K$ and if its Minkowski functional 
\begin{equation} \label{gauge}
\|x\|_K=\inf\{\lambda >0 : x\in \lambda K\}
\end{equation} is continuous on $\mathbb{R}^d$.

For both symmetric star and convex bodies the gauge function (\ref{gauge}) is $1$-homogeneous and continuous and $K=\{x\in \mathbb{R}^d : \|x\|_K\leq 1\}$. It is a norm in the case of a symmetric convex body and when $K$ is star body it is a quasi-norm, i.e. $\|x+y\|_K\leq C(\|x\|_K+\|y\|_K)$ for all $x,y\in \mathbb{R}^d$.

For any star body $K\subset \mathbb{R}^d$ we define the cone (probability) measure $\mu_K$ on $\partial K$ as follows
\[
\mu_K(A) =\frac{|\{tx : t\in[0,1] , x\in A\}|}{|K|}
\]
for $A\subset \partial K$.\par 
The cone measure is also characterized by the following formula that holds for all $f\in L_1(\mathbb{R}^d,dx)$
\begin{equation}\label{eq: polar integration formula}
\int_{\mathbb{R}^d}f(x)dx = d|K|\int_{\partial K}\int_0^\infty f(r\theta) r^{d-1}drd\mu_K(\theta).
\end{equation}

\subsubsection{The body $B_p^n(K)$}Let $K$ be a symmetric star body on $\mathbb{R}^d$ and $\|\cdot\|_K$ be the quasi-norm that it induces. For $p\in (0,+\infty)$ we equip $\mathbb{R}^{nd}$ with the quasi-norm (in general) $\|\cdot\|_{l_p^n(K)}$ defined as $\|x\|^p_{l_p^n(K)}=\sum_{i=1}^n\|x_i\|_K^p$ and denote its unit ball by $B_p^n(K)$ and $S_p^n(K):=\partial B_p^n(K)$. This body will concern us in Section \ref{Section High dim}. For future use, we calculate its volume 
\begin{equation} \label{volume of }
    |B_p^n(K)|= \frac{|K|^n\Gamma(1+\frac{d}{p})^n}{\Gamma(1+\frac{nd}{p})}.
\end{equation}
Indeed, using (twice) the known formula 
$
\int_{\mathbb{R}^N}\exp(-\|x\|_L^q)dx=|L|\Gamma\left(\frac{N}{q}+1\right)
$, we get
\[
|B_p^n(K)|\Gamma\left(\frac{nd}{p}+1\right)=\int_{\mathbb{R}^{nd}}e^{-\|x\|_{l_p^n(K)}^p}dx=\left( \int_{\mathbb{R}^d}e^{-\|x\|_K^p}dx   \right)^n=\left( |K|\Gamma\left(\frac{d}{p}+1\right)   \right)^n.
\]

\subsection{Subspaces of $L_p$ and intersection bodies} In a subsequent section, in order to deal with certain star bodies, we will use results about embeddings of finite dimensional quasi-normed spaces $(\mathbb{R}^d,\|\cdot\|)$ in $L_p$, for $p\in (0,2]$ or $p\in (-d,0)$. For background on this subject and much more we refer to the book \cite{koldobskybook} by Koldobsky where most of the material below is taken from. See also \cite{embext}.

\subsubsection{}When $p>0$, we write $L_p$ for the space $L_p([0,1])$ equipped with the usual quasi-norm denoted by $\|\cdot\|_{L_p}$. The quasi-normed space $(\mathbb{R}^d,\|\cdot\|)$ embeds isometrically in $L_p([0,1])$ if there exists a linear isometry $S: (\mathbb{R}^d,\|\cdot\|)\to (L_p([0,1]),\|\cdot\|_{L_p})$. Recall that for $u,x\in \mathbb{R}^d$, $u\otimes u(x):=\langle u,x\rangle u$. A Borel measure $\nu$ on the sphere $S^{d-1}$ is called isotropic if 
\begin{equation}
    I_d= \int_{S^{d-1}} u\otimes u \ \nu(du).
\end{equation}

For all $p>0$ a useful characterization by Lewis(\cite{lewis}) and by Schechtmann and Zvavitch(\cite{SchZv}), (see also \cite[Lemma 6.4]{koldobskybook} for a first representation of the norm by Lévy) tells that the space $(\mathbb{R}^d,\|\cdot\|)$ embeds isometrically to $L_p$ if and only if there exists an invertible linear map $T$ and an isotropic measure $\nu$ on $S^{d-1}$ such that $\|Tx\|^p=\int_{S^{d-1}}|\langle x,u \rangle|^pd\nu(u)$. In that case if $K:= \{x\in \mathbb{R}^d : \|x\|\leq 1 \}$, it is said that $T$ puts $K$ in Lewis' position.

Finally, for $p\in (0,2]$, the quasi-normed space $(\mathbb{R}^d,\|\cdot\|)$ embeds isometrically in $L_p$ if and only if the (continuous) function $h(x)=\exp(-\|x\|^p)$ is positive definite, see also \cite{lois}.

\subsubsection{}For negative values of $p$ the above definition of embedding does not make sense. In that case Koldobsky gave the following definition:

\begin{definition} \label{definition embed negative}
    Let $l\in(0,d)$. The quasi-normed space $(\mathbb{R}^d, \|\cdot\|)$ embeds in $L_{-l}$, if there exists a finite Borel
measure $\sigma_l$ on $S^{d-1}$, such that for every even $\psi \in \mathcal{S}(\mathbb{R}^d)$ it holds that 
\begin{equation} \label{eq for embeddings to -l}
    \int_{\mathbb{R}^d}\frac{1}{\|x\|^{l}}\psi(x)dx = \int_0^\infty \int_{S^{d-1}}t^{l-1}\Hat{\psi}(t\theta)d\sigma_l(\theta)dt.
\end{equation}
\end{definition}

According to \cite[Corollary 2.26]{koldobskybook}, this is equivalent to the fact that $\|\cdot\|^{-l}$ represents a positive definite distribution. The motivation for this general definition comes from the analogous property that the intersection bodies, which were mentioned in the introduction, have. 

In general, if $k\in \{1,...,d-1\}$ and $L\subset\mathbb{R}^d$ is a symmetric star body, then a symmetric star body $K$ is called the $k$-intersection body of $L$, if for any $k$ co-dimensional subspace $H$ of $\mathbb{R}^d$ it holds 
\[
|K\cap H^\perp|=|L\cap H|.
\]
A calculation shows that in the sense of distributions on has $\left(\|\cdot\|_K^{-k}\right)^\wedge(\theta) =c\|\theta\|_L^{k-d}, \: \theta \in S^{d-1}$, up to some constant $c=c_{d,k}$, extended as $(k-d)$- homogeneous function.
In particular, for every even $\psi \in \mathcal{S}(\mathbb{R}^d)$ one has
\[
\int_{\mathbb{R}^d}\|x\|_K^{-k}\psi(x)dx=\int_{\mathbb{R}^d} \left(\|\cdot\|_K^{-k}\right)^\wedge(x) \hat{\psi}(x)dx= \int_{S^{d-1}} \int_0^{+\infty} t^{k-1}\hat{\psi}(t\theta)dt\ c_{d,k}\|\theta\|^{k-d}_L d\theta,
\]
which explains the reason behind Definition \ref{definition embed negative}.

The class of examples of star bodies $K$ for which $(\mathbb{R}^d, \|\cdot\|_K)$ embeds to $L_{-l}$ (for each exponent $l$ and dimension $d$) is big. In particular, by \cite[Corollary 4.9 and Theorem 4.11]{koldobskybook} one has the following examples :\par
\begin{enumerate}
 \item For all $K$ that are $k$-intersection bodies, $\frac{1}{\|x\|_K^{k}}$ is positive definite, and thus $(\mathbb{R}^d, \|\cdot\|_K)$ embeds to $L_{-k}$.
 \item  Any finite dimensional subspace of $L_p$, for $p\in (0,2]$ embeds in $L_{-l}$ for every  $0< l < d$.\par
\item For \textit{each} star body $K$ in $\mathbb{R}^d$ and $l\in [d-3,d)$ the function $\frac{1}{\|x\|_K^{l}}$ is positive definite. 
\end{enumerate}

\subsection{Stable laws and gaussian mixtures} \label{subsection mixtures completely monnot etc}

A random variable $X$ is called a Gaussian mixture if it has the same distribution as the product of two independent random variables $Y$ and $Z$, with $Y$ being almost surely non-negative and $Z$ a (standard) Gaussian random variable. Recently Eskenazis, Nayar and Tkocz \cite{ENT} studied extensively these random variables in connection with geometric questions. A useful property is that a symmetric random variable $X$ with density $\phi$ is a Gaussian
mixture if and only if the function $x \mapsto \phi(\sqrt{x})$ is completely monotonic on $(0,+\infty)$. A real function $h\in C^\infty((0,+\infty))$ is called completely monotonic if it satisfies $(-1)^nh^{(n)}(t)\geq 0$, for all $n\in \mathbb{N}$ and $t > 0$. According to Bernstein's theorem a function $h$ is completely monotonic if and only if it is the Laplace transform of a non-negative Borel measure $\nu$ on $[0,+\infty)$, i.e.
\[
h(t) = \int_0^\infty e^{-tx} \nu(dx).
\]

Examples of Gaussian mixtures include the $p$-stable random variables and the random variables with density proportional to $e^{-|x|^p}$, for $p\in (0,2]$. For $p\in (0,2]$, by (symmetric) $p$-stable random variable we mean a random variable $X$ whose density $\phi$ has Fourier transform $\Hat{\phi}(x)=\exp(-c_p|x|^p)$, see \cite{univariatestabledistributions}.

For $p>2$, $p$-stable random variables do not exist. However, Oleszkiewicz in \cite{pseudostable} introduced a new class of random variables that he called pseudo-stable random variables. For $p > 2$ a (symmetric) random variable $X$ is called $p$-pseudo-stable if $X$ is not Gaussian and for any $a,b\in \mathbb{R}$ and $X_1, X_2$ independent $p$-pseudo-stable random variables copies of $X$, there exists $v(a, b)\in \mathbb{R}$ such that
\[
aX_1 + bX_2 =_d (|a|^p + |b|^p)^{\frac{1}{p}}X + v(a, b)Z,\]
where $Z$ is a standard Gaussian independent of $X$.

What will be crucial for us is that the Fourier transform a $p$-pseudo-stable random variable $X$ is
\[
\mathbb{E}e^{-2\pi itX}= \exp(-c_1|t|^p-c_2t^2),
\]
where $c_i>0$. Finally, Oleszkiewicz proved that these random variables exist precisely when $p$ lies in the set $\cup_{n=1}^\infty (4n,4n+2)$, while their second moment is always finite. See \cite[Theorem 1]{pseudostable} for details. More generally, Misiewicz and Mazurkiewicz defined, in \cite{firstpc_pseudostable}, for $r\in (0,2)$ (i.e. having a $r$-stable instead of the Gaussian in the above definition), the $(p,r)$-pseudo-stable random variables in the spirit of Oleszkiewicz and proved existence for a certain range of parameters. These are symmetric random variables that have Fourier transform of the form $\exp(-c_1|t|^p-c_2|t|^r)$. See also \cite{c_dpseudostable}.

The vectorial analogue of a Gaussian mixture was introduced in \cite{eskgav}:  a random vector $X$ in $\mathbb{R}^d$ is a Gaussian mixture if $X$ has the same distribution as $YZ$, where $Y$ is a $d\times d$ random matrix that is almost surely positive definite and $Z$ is a standard Gaussian random vector independent of $Y$. The density $\phi$ of a Gaussian mixture $X$ on $\mathbb{R}^d$ (for $d\geq 1$) is thus given by 
\begin{equation} \label{density of GM}
    \phi(x)=\mathbb{E}_Y \left[\frac{1}{\det(\sqrt{2\pi }Y)}\exp(-\frac{1}{2}|Y^{-1}x|^2)\right].
\end{equation}
 
Finally, the vector analogue for $p$-stable exists as well ($p\in (0,2)$). We refer to \cite{nolan:book2} and \cite{stablebook} for a complete description. We are only interested in symmetric $p$-stable random vectors $X$ on $\mathbb{R}^d$, which are characterized by the existence of a finite symmetric measure $\sigma_0$ on $S^{d-1}$ such that 
\[
\mathbb{E}e^{-2\pi i \langle t,X\rangle}= \exp\left(-\int_{S^{d-1}}|\langle t,\theta
 \rangle|^p\sigma_0(d\theta)  \right).
\]
The measure $\sigma_0$ is called spectral measure.

\section{Product measures and Fourier} \label{Section main}

\subsection{The section function} \label{section sections} 

In this section we consider a high-dimensional generalization of the question posed in the introduction about sections of product measures. 

We fix $d,n\in \mathbb{N}$. For $t\in \mathbb{R}^d$ and $y\in \mathbb{R}^n$, we denote by
\[
H_{y,t}:=\{x=(x_1,\ldots,x_n)\in \mathbb{R}^{nd} : y_1x_1+\cdots+y_nx_n=t\},
\]
the $(nd-d)$-dimensional affine block subspace associated to the pair $(y,t)$. When $t=0$ we will write $H_y$. Also for $y\in \mathbb{R}^n$ and $x\in \mathbb{R}^{nd}$ we denote $\langle y,x\rangle:=y_1x_1+\cdots+y_nx_n\in \mathbb{R}^d.$

Let now $\mu(dx)=\phi(x)dx$ be a probability measure on $\mathbb{R}^d$. We will always assume that $\phi$ is even, has integrable Fourier transform and $\hat{\phi}\geq 0$. We define the (block) section function $S_{\mu^n}:\mathbb{R}^n\setminus \{0\}\times\mathbb{R}^d \to [0,+\infty)$ of the $n$-fold product probability measure $\mu^n$ on $\mathbb{R}^{nd}$ as
\[
S_{\mu^n}(y,t):=\int_{H_{y,t}}\prod_{k=1}^n\phi(x_k)\ \mathcal{H}^{nd-d}(dx),
\]
and we will denote $S_{\mu^n}(y):=S_{\mu^n}(y,0)$. Inspired by Ball's representation of the sections of the cube (which corresponds to $d=1$ and $\mu(dt)=\mathbbm{1}_{[-1/2,1/2]}(t)dt$), we use Fourier inversion  to give a representation of the section function $S_{\mu^n}$. More precisely, for any $y\neq 0$, taking the Fourier transform  with respect to the variable $t\in \mathbb{R}^d$, we get
\[
\int_{\mathbb{R}^d} e^{-2\pi i\langle s,t\rangle}\int_{\langle x,y\rangle = t} \prod_{k=1}^n\phi(x_k)\ \mathcal{H}^{nd-d}(dx)\ dt=|y|^d\int_{\mathbb{R}^{nd}}   e^{-2\pi is\langle x,y\rangle}\prod_{k=1}^n\phi(x_k)\ dx=|y|^d\prod_{k=1}^n\hat{\phi}(y_ks),
\]
where we applied the coarea formula for the function $F_y:\mathbb{R}^{nd}\to \mathbb{R}^d$ defined as $F_y(x_1,\ldots,x_n)=y_1x_1+\cdots+y_nx_n$, and thus
\begin{equation}\label{fourierofsec}
\Hat{S}_{\mu^n}(y,s)= |y|^d\prod_{k=1}^n\hat{\phi}(y_ks) \ .
\end{equation}
Observe that from the assumptions on $\phi$, $\Hat{S}_{\mu^n}(y,s)$ is integrable in $s$, for any $y\in \mathbb{R}^n$.  Inverting the Fourier transform, yields
\[
S_{\mu^n}(y,t)=|y|^d\int_{\mathbb{R}^d} e^{2\pi i\langle s,t\rangle}\prod_{k=1}^n\hat{\phi}(y_ks)ds,
\]
and in particular
\[
S_{\mu^n}(y)=S_{\mu^n}(y,0)= |y|^d\int_{\mathbb{R}^d}\prod_{k=1}^n\hat{\phi}(y_ks)ds.
\]

In analogy with the classical question on sections of convex bodies, we study for which directions $\theta\in S^{n-1}$, $S_{\mu^n}(\theta)$ is extremal. Since $\theta=(\theta_1,\ldots,\theta_n)$
belongs to the unit sphere, setting $q_k=\theta_k^2$, we have $(q_1,\ldots,q_n)\in \Delta^{n-1}:= \{(a_1,\ldots,a_n)\in [0,1]^n\ |\ a_1+\cdots+a_n=1\} $. In particular,
\[
S_{\mu^n}(\sqrt{q_1},...,\sqrt{q_n})= \int_{\mathbb{R}^d}\prod_{k=1}^n\hat{\phi}(\sqrt{q_k}s)ds .
\]

Combining the above calculation with the elementary Lemma \ref{Schurlemma} gives:

\begin{proposition} \label{sectionlogsth}
    Let $\mu$ be a probability measure on $\mathbb{R}^d$ with density $\phi$. Assume that $\phi$ is even and also that its Fourier transform is positive and integrable. Consider the function 
\[
 G(q):=S_{\mu^n}(\sqrt{q}), \ \text{for } \ q\in \Delta^{n-1},
\]
 where $\sqrt{q}:=(\sqrt{q_1},\ldots,\sqrt{q}_n)$. Then:
    
1) If for each $t\in \mathbb{R}^d$ the function $ r\mapsto\Hat{\phi}(\sqrt{r}t)$ is log-convex on $[0,+\infty)$, then $G$ is log-convex.

2) If for each $t\in \mathbb{R}^d$ the function $ r\mapsto\Hat{\phi}(\sqrt{r}t)$ is log-concave on $[0,+\infty)$, then $G$ is Schur concave.

\end{proposition}

\begin{proof}

The first assertion follows by Hölder. Indeed, suppose that $\Hat{\phi}(\sqrt{\cdot}t)$ is log-convex, then if $q,r \in \Delta^{n-1}$ and $\lambda\in (0,1)$, we have
    \begin{align*}
    S_{\mu^n}(\sqrt{\lambda q+(1-\lambda)r})=\int_{\mathbb{R}^d}\prod_{k=1}^n\hat{\phi}(\sqrt{\lambda q_k+(1-\lambda) r_k}t)dt \leq \int_{\mathbb{R}^d}\prod_{k=1}^n\hat{\phi}(\sqrt{q_k}t)^\lambda \hat{\phi}(\sqrt{r_k}t)^{1-\lambda} dt \\ \leq S_{\mu^n}(\sqrt{q})^{\lambda}S_{\mu^n}(\sqrt{r})^{1-\lambda},
    \end{align*}
    which is the log-convexity of $G$.

For the second point, since $\Hat{\phi}(\sqrt{\cdot}t)$ is log-concave, Lemma \ref{Schurlemma} gives that $\Phi(a)=\prod_{k=1}^n\hat{\phi}(\sqrt{a_k}t)$ is Schur concave, for every $t\in \mathbb{R}^d$. Let $q\prec r$ in $\Delta^{n-1}$ we have 

\[
S_{\mu^n}(\sqrt{q}) = \int_{\mathbb{R}^d}\prod_{k=1}^n\hat{\phi}(\sqrt{q_k}t)dt \geq \int_{\mathbb{R}^d}\prod_{k=1}^n\hat{\phi}(\sqrt{r_k}t)dt= S_{\mu^n}(\sqrt{r})
\]
    which is the desired property.
    
\end{proof}

In particular, as log-convexity implies convexity and thus Schur convexity as well (since the section function is invariant under permutations) the above proposition gives the following about the extremal (block) sections of such product measures, complementing for certain measures (and hyperplanes) \cite[Proposition 18]{bartheextremal}:

\begin{corollary}
    Let $\mu(dx)=\phi(x)dx$ be as in Proposition \ref{sectionlogsth}. If for each $t\in \mathbb{R}^d$ the function $[0,\infty)\ni r\mapsto\Hat{\phi}(\sqrt{r}t)$ is log-convex, then for each $\theta\in S^{n-1}$
    \[
    S_{\mu^n}\left(\frac{1}{\sqrt{n}},\ldots,\frac{1}{\sqrt{n}}\right)\leq  S_{\mu^n}(\theta)\leq  S_{\mu^n}\left(1,0,\ldots,0\right)
    \]
    and the converse in the log-concave case.
\end{corollary}

\par \vspace{1em}
When $d=1$, the log-convexity of $\Hat{\phi}(\sqrt{\cdot})$
is an assumption that has appeared in many places in the literature in connection with Schur monotonicity, starting from \cite{koldobsky98} where it was used
to find the extremal sections of $p$-balls for $p\in (0,2)$. Later, the condition appeared again in a paper by Barthe and Naor \cite{BN}, to find the extremal \textit{projections} of $p$-balls for $p\in (2,\infty)$, and more recently in \cite{ENT} as a property of Gaussian mixtures.

On the other hand the condition in the log-concave case seems quite strong, and in fact the first trivial example is the Gaussian measure. A more involved class of examples comes from the pseudo-stable random variables, which we introduced in Section \ref{subsection mixtures completely monnot etc}. In Section \ref{section examples} we will expand more on the examples of measures.

\subsection{The general argument}
We consider now the problem of finding the extremals of certain functionals, associated to product measures defined on $\mathbb{R}^d$. The following is the main result of the present note:

\begin{theorem}\label{centralthm} Let $q>0$. Let $\mu(dx)=\phi(x)dx$ be a Borel probability measure on $\mathbb{R}^d$, with $\phi$ even and $\hat{\phi}$ positive and  integrable. Let $h:\mathbb{R}^d\to \mathbb{R}$ be a continuous positive definite function. Let $n\in \mathbb{N}$, $X_1,\ldots,X_n$ be i.i.d. random vectors distributed according to $\mu$ and consider the function  $H: [0,+\infty)^n \to [0,+\infty]$ that is defined as
\[
H(a):=\mathbb{E}h\left(  a_1^{\frac{1}{q}}X_1+\cdots+ a_n^{\frac{1}{q}}X_n \right).
\]
Then:\par
1) If for each $t\in \mathbb{R}^d$ the function $[0,+\infty)\ni r\mapsto\Hat{\phi}(r^{\frac{1}{q}}t)$ is log-convex , then $H$
is log-convex.

Conversely, if for each positive definite function $h$ on $\mathbb{R}^d$, the associated functional $H$ is log-convex, then for each $t\in \mathbb{R}^d$ the function $[0,+\infty)\ni r\mapsto\Hat{\phi}(r^{\frac{1}{q}}t)$ is log-convex.

2) If for each $t\in \mathbb{R}^d$ the function $[0,+\infty)\ni r\mapsto\Hat{\phi}(r^{\frac{1}{q}}t)$ is log-concave, then $H$ is Schur concave.

Conversely, if for each positive definite function $h$ on $\mathbb{R}^d$, the associated functional $H$ is Schur concave and in addition  $[0,\infty)\ni r\mapsto\Hat{\phi}(r^{\frac{1}{q}}t)$ is continuous then for each $t\in \mathbb{R}^d$ the function $[0,\infty)\ni r\mapsto\Hat{\phi}(r^{\frac{1}{q}}t)$ is log-concave.
\end{theorem}

\begin{proof}
    Denote $a^{\frac{1}{q}}:=( a_1^{\frac{1}{q}},\ldots,a_n^{\frac{1}{q}})$. We re-write the function $H$ as    
\[
H(a)=\mathbb{E}h\left(  a_1^{\frac{1}{q}}X_1+\cdots+ a_n^{\frac{1}{q}}X_n \right)= \int_{\mathbb{R}^{nd}} h\left(\langle a^{\frac{1}{q}},x \rangle\right)\prod_{i=1}^n \phi(x_i) \ dx, \] and use the co-area formula with $F(x_1,\ldots,x_n)=a_1^{\frac{1}{q}}x_1+\cdots+a_n^{\frac{1}{q}}x_n$ to get
\[H(a)=\frac{1}{|a^{\frac{1}{q}}|^d}\int_{\mathbb{R}^d} h(t) \int_{\langle a^{\frac{1}{q}},x \rangle=t} \prod_{i=1}^n \phi(x_i) \ \mathcal{H}^{nd-d}(dx)\ dt.
\]
Since $h$ is positive definite, Bochner's theorem provides a finite positive Borel measure $\nu$ on $\mathbb{R}^d$ such that $h=\hat{\nu}$ and so 
\[
H(a)= \frac{1}{|a^{\frac{1}{q}}|^d}\int_{\mathbb{R}^d} \int_{\mathbb{R}^d} e^{-2\pi i\langle r,t\rangle}\nu(dr) \int_{\langle a^{\frac{1}{q}},x \rangle=t} \prod_{i=1}^n \phi(x_i)\ \mathcal{H}^{nd-d}(dx)\ dt.
\]

From the relation (\ref{fourierofsec}) we get
\[
H(a)=\int_{\mathbb{R}^d} \prod_{i=1}^n\hat{\phi}(a_i^{\frac{1}{q}}s)d\nu(s).
\]
Suppose now that $[0,+\infty)\ni r\mapsto \hat{\phi}(r^{\frac{1}{q}}s)$ is log-convex for each $s\in \mathbb{R}^{d}$. Let $a,b\in [0,+\infty)^n$ and $\lambda\in (0,1)$, then using Hölder we get
\begin{align*}
H(\lambda a+(1-\lambda)b)=\int_{\mathbb{R}^d} \prod_{i=1}^n\hat{\phi}((\lambda a_i+(1-\lambda)b_i)^{\frac{1}{q}}s)d\nu(s) \leq \int_{\mathbb{R}^d} \prod_{i=1}^n\hat{\phi}(a_i^{\frac{1}{q}}s)^\lambda\prod_{i=1}^n\hat{\phi}(b_i^{\frac{1}{q}}s)^{1-\lambda} d\nu(s)\\ \leq \left(\int_{\mathbb{R}^d} \prod_{i=1}^n\hat{\phi}(a_i^{\frac{1}{q}}s)d\nu(s)\right)^{\lambda}\left(\int_{\mathbb{R}^d} \prod_{i=1}^n\hat{\phi}(a_i^{\frac{1}{q}}s)d\nu(s)\right)^{1-\lambda}=H(a)^\lambda H(b)^{1-\lambda}
\end{align*}
and so $H$ is log-convex on $[0,+\infty)^n$.

Conversely, suppose that for each positive definite function $h$, the associated functional $H$ is log-convex. Then, let $x_0\in \mathbb{R}^d$ and for $\varepsilon > 0$ consider the Borel (probability) measure $\nu_\varepsilon$ defined via

\[\nu_\varepsilon(dx)=\frac{1}{2|B(x_0,\varepsilon)|}\mathbbm{1}_{B(x_0,\varepsilon)\cup B(-x_0,\varepsilon)}(x)dx\]
and let $h=\hat{\nu}_\varepsilon$.

Let  $\lambda\in (0,1)$ and take $a=(a,0,\ldots,0)$ and $b=(b,0,\ldots,0)$, then by the log-convexity of $H$ and the representation $H(a_1,\ldots,a_n)=\int \prod_{i=1}^n\hat{\phi}(a_i^{\frac{1}{q}}s)d\nu_\varepsilon(s)$ we see that
\[
\frac{1}{|B(x_0,\varepsilon)|}\int_{B(x_0,\varepsilon)} \hat{\phi}((\lambda a+(1-\lambda)b)^{\frac{1}{q}}s)ds \leq \left(\frac{1}{|B(x_0,\varepsilon)|}\int_{B(x_0,\varepsilon)} \hat{\phi}(a^{\frac{1}{q}}s)ds\right)^{\lambda}\left(\frac{1}{|B(x_0,\varepsilon)|}\int_{B(x_0,\varepsilon)} \hat{\phi}(b^{\frac{1}{q}}s)ds\right)^{1-\lambda}
\]
where we used that the density $\phi$, and thus its Fourier transform is even.

Now by Lebesque's differentiation theorem, letting $\varepsilon\to 0^+$, we get
\[
\hat{\phi}((\lambda a+(1-\lambda)b)^{\frac{1}{q}}x_0)\leq \hat{\phi}(a^{\frac{1}{q}}x_0)^{\lambda}\hat{\phi}(b^{\frac{1}{q}}x_0)^{1-\lambda}.
\]

Since this holds for any $a,b>0$ and $x_0\in \mathbb{R}^d$ we get that the function $(0,+\infty)\ni r\mapsto \hat{\phi}(r^{\frac{1}{q}}s)$ is log-convex for each $s\in \mathbb{R}^d$.

In the log-concave case, we use again Lemma \ref{Schurlemma} as in the Proposition \ref{sectionlogsth} for sections, to see that for each $s\in \mathbb{R}^d$ the function 
\[
(0,+\infty)^n\ni (a_1,\ldots,a_n) \mapsto \prod_{m=1}^n\hat{\phi}(a_m^{\frac{1}{q}}s)
\]
 is Schur concave by the log-concavity of $\hat{\phi}((\cdot)^{\frac{1}{q}}s)$. Using the representation \[H(a_1,\ldots,a_n)=\int \prod_{i=1}^n\hat{\phi}(a_i^{\frac{1}{q}}s)d\nu(s)\] we can infer that $H$ is Schur concave.
 
The proof of the converse direction is identical to the log-convex case. We apply the condition for $h_\varepsilon$ as above and if $a\prec b$ we have
\[
H_\varepsilon(a)\geq H_\varepsilon(b)
\]
now letting $\varepsilon\to 0^+$ we get that for each $s\in \mathbb{R}^d$ the function $F(a_1,\ldots,a_n)=\prod_{m=1}^n\hat{\phi}(a_m^{\frac{1}{q}}s)$ is Schur concave and we conclude with Lemma \ref{Schurlemma}.

\end{proof}

When the random vectors are distributed uniformly on the sphere and $q=2$, in \cite{culverhouse} the authors proved that the Schur monotonicity of $H$ is equivalent to the bisubharmonicity of $h$, with the notation as above. For $d=1$, other results close to that spirit have appeared also in \cite{Pinelis},\cite{Figieletc} and \cite{ENTSharpcomparison}.

\begin{remark} The conclusion of Theorem \ref{centralthm} in the log-convex scenario is stronger than the Schur convexity of $H$, since the functional is symmetric under permutation in its arguments log-convexity implies convexity. 
In analogy to the problems posed in the introduction, the Schur monotonicity of some function defined as in Theorem \ref{centralthm} gives the extremals of the functional while $\theta\in S_q^{n-1}:= \{x\in \mathbb{R}^n : \|x\|_q=1\}$.

Note that if a random vector $X$ on $\mathbb{R}^d$ is distributed according to a measure $\mu$, satisfying either of the above assumptions, for each $y\in \mathbb{R}^d$, the r.v. $\langle X,y\rangle$ also satisfies the corresponding assumption.

Notice that there is no assumption on the sign of $h$, while the fact that we considered probability measures is just a normalization and below we may forget this if we deal with finite measures. Moreover, the family of functions $N_q(t)=t^{\frac{1}{q}}$ that we considered could be replaced by any function $N$, requiring a condition on $\hat{\phi}(N(\cdot)t)$. In lack of any interesting examples we avoid such generality.

Finally, it is clear that the proof is identical when the function $h$ is in $L^2(\mathbb{R}^d)$, as Theorem \ref{mainthmintro} was stated in the introduction.
\end{remark}

The main theorem can be useful in moment comparison:

\begin{corollary} \label{momentcorollary}
    Let $q>0$ and $X_1,\ldots,X_n$ be i.i.d. random variables on $\mathbb{R}$ distributed according to $\mu(dx)=\phi(x)dx$, where $\phi$ is even with integrable and positive Fourier transform $\hat{\phi}$ satisfying $\Hat{\phi}(t^\frac{1}{q})$ is log-convex (resp. log-concave) on $(0,+\infty)$.
    
Then, for each $\lambda>0$ and $p\in (0,2]$, the function 
\begin{equation} \label{boosted 1 dim main}
(0,+\infty)^n \ni a \mapsto \mathbb{E} e^{-\lambda|\sqrt[q]{a_1}X_1+\cdots+ \sqrt[q]{a_1}X_n|^p}
\end{equation}
is log-convex (resp. Schur concave).
    
    In particular, assuming that   $\mathbb{E} |X|^p<+\infty$, for some $p\in (0,2]$ the function 
\[
(0,+\infty)^n \ni a \mapsto \mathbb{E} |a_1^{\frac{1}{q}}X_1+\cdots+ a_n^{\frac{1}{q}}X_n|^p
\]
is concave (resp. Schur convex), while if $p\in (0,1)$ and $\mathbb{E} |X|^{-p}<+\infty$ then the function
\[
(0,+\infty)^n \ni a \mapsto \mathbb{E} |a_1^{\frac{1}{q}}X_1+\cdots+ a_n^{\frac{1}{q}}X_n|^{-p}
\]
is log-convex (resp. Schur concave).
\end{corollary}

\begin{proof}

Recall first that for each $\lambda>0$ and $p\in (0,2]$ the function $h_\lambda(t)=e^{-\lambda|t|^p}$ is positive definite. 

Suppose that $\Hat{\phi}((\cdot)^\frac{1}{q})$ is log-convex.
Applying Theorem \ref{centralthm} for each $h_\lambda$ for $\lambda >0$, and $\mu$, we get that the function
\[
(0,+\infty)^n \ni a \mapsto \mathbb{E} \exp\left(-\lambda|a_1^{\frac{1}{q}}X_1+\cdots+ a_n^{\frac{1}{q}}X_n|^p\right) 
\]
is log-convex. 

Assume now that the measure satisfies  $\int_{\mathbb{R}}|x|^p\mu(dx)<\infty$.  Using convexity one gets, for $a,b \in (0,+\infty)^n$ and $t\in (0,1)$ that
\[
\mathbb{E} \exp\left(-\lambda|(ta_1+(1-t)b_1)^{\frac{1}{q}}X_1+\cdots+ (ta_n+(1-t)b_n)^{\frac{1}{q}}X_n|^p\right)\]\[ \leq t\mathbb{E} \exp\left(-\lambda|a_1^{\frac{1}{q}}X_1+\cdots+ a_n^{\frac{1}{q}}X_n|^p\right)+(1-t)\mathbb{E} \exp\left(-\lambda|b_1^{\frac{1}{q}}X_1+\cdots+ b_n^{\frac{1}{q}}X_n|^p\right)
\]
Since the two sides agree at $\lambda=0$ taking the derivative we get 
\[
\mathbb{E} |(ta_1+(1-t)b_1)^{\frac{1}{q}}X_1+\cdots+ (ta_n+(1-t)b_n)^{\frac{1}{q}}X_n|^p\]\[ \geq t\mathbb{E} |a_1^{\frac{1}{q}}X_1+\cdots+ a_n^{\frac{1}{q}}X_n|^p+(1-t)\mathbb{E} |b_1^{\frac{1}{q}}X_1+\cdots+ b_n^{\frac{1}{q}}X_n|^p
\]
and so the function 
\[
(0,+\infty)^n \ni a \mapsto \mathbb{E} |a_1^{\frac{1}{q}}X_1+\cdots+ a_n^{\frac{1}{q}}X_n|^p
\]
is concave, giving the second claim.

For the last part the argument is similar but now integrating with respect to the parameter $\lambda$. Indeed, let $a,b \in (0,+\infty)^n$ and $t\in (0,1)$, by log-convexity and then Hölder we have

\[
\int_0^\infty \mathbb{E} \exp\left(-\lambda|(ta_1+(1-t)b_1)^{\frac{1}{q}}X_1+\cdots+ (ta_n+(1-t)b_n)^{\frac{1}{q}}X_n|^p\right) d\lambda\]\[ \leq \int_0^\infty\left[\mathbb{E} \exp\left(-\lambda|a_1^{\frac{1}{q}}X_1+\cdots+ a_n^{\frac{1}{q}}X_n|^p\right)\right]^t\left[\mathbb{E} \exp\left(-\lambda|b_1^{\frac{1}{q}}X_1+\cdots+ b_n^{\frac{1}{q}}X_n|^p\right)\right]^{1-t}d\lambda
\]
\[
\leq \left[\int_0^\infty\mathbb{E} \exp\left(-\lambda|a_1^{\frac{1}{q}}X_1+\cdots+ a_n^{\frac{1}{q}}X_n|^p\right)d\lambda\right]^t\left[\int_0^\infty\mathbb{E} \exp\left(-\lambda|b_1^{\frac{1}{q}}X_1+\cdots+ b_n^{\frac{1}{q}}X_n|^p\right)d\lambda\right]^{1-t}.
\]
 This is now equivalent to 
\[
\mathbb{E} |(ta_1+(1-t)b_1)^{\frac{1}{q}}X_1+\cdots+ (ta_n+(1-t)b_n)^{\frac{1}{q}}X_n|^{-p} 
\]
\[
\leq \left[\mathbb{E} |a_1^{\frac{1}{q}}X_1+\cdots+ a_n^{\frac{1}{q}}X_n|^{-p}\right]^t\left[\mathbb{E}|b_1^{\frac{1}{q}}X_1+\cdots+ b_n^{\frac{1}{q}}X_n|^{-p}\right]^{1-t}
\]
which is the last claim. For the log-concave case one proceeds similarly using the Schur concavity instead.
\end{proof}

Considering the variation of the parameter $p$ as $p\to 0$, differentiation can also give the Schur monotonicity of the following function (provided the integrals are convergent)
\[
(0,+\infty)^n \ni a \mapsto \mathbb{E} \log\Big|a_1^{\frac{1}{q}}X_1+\cdots+ a_n^{\frac{1}{q}}X_n\Big|.
\]
The details are omitted.

While the moment comparison is indeed the usual interest, Theorem \ref{centralthm} contains many other examples. An unusual case at first sight is for example $h(x)=\cos x$ or even $\sum_i \cos(\alpha_ix)$, and gives information about

\[
H(a_1,\ldots,a_n)=\mathbb{E}\cos(\sqrt[q]{a_1}X_1+\cdots+\sqrt[q]{a_n}X_n).
\]

\subsection{Examples} \label{section examples}

 $\bullet$ \textbf{Case $q=2$}. This is the case that is the most interesting so far since it corresponds to the study of extrema of functionals on the Euclidean sphere. 

In one dimension, the hypothesis of the main theorem boils down to $\hat{\phi}(\sqrt{\cdot})$ being log-concave or log-convex on $(0,+\infty)$. A first example is given by the Gaussian measure for which this function is log-linear. The log-concavity hypothesis is satisfied by the $p$-pseudo-stable random variables, for each $p\in \cup_{n=1}^\infty (4n,4n+2)$, since their Fourier transform equals $\hat{\phi}(t)=\exp(-c_1|t|^p-c_2t^2)$, and thus $log(\hat{\phi}(\sqrt{t}))=-c_1t^{\frac{p}{2}}-c_2t$.

For the log-convexity hypothesis, examples include the Gaussian mixtures, that in fact satisfy the stronger that $\Hat{\phi}(\sqrt{\cdot})$ is completely monotonic. An example of density that does not correspond to a Gaussian mixture was pointed out to us by Piotr Nayar and is given by $\phi(t)=2\left(1-|t|\right)_+^3$.

In higher dimensions, the vector Gaussian mixtures (see Section \ref{subsection mixtures completely monnot etc}), still satisfy the log-convexity hypothesis, which amount to $\hat{\phi}(\sqrt{\cdot}\ t):[0,+\infty)\to (0,+\infty)$ being log-convex for each $t\in \mathbb{R}^d$. Indeed, if we denote the density of the random vector $X$ by $\phi$, and $X=_d YZ$, we get
     \begin{equation} \label{fourier of GM}
     \Hat{\phi}(x)= \mathbb{E}_Y\left[e^{-2\pi^2|Y^{t}x|^2}\right]
     \end{equation}
     Setting $x=\sqrt{r}t$ we get $\Hat{\phi}(\sqrt{r}\ t)=\mathbb{E}_Y\left[\exp\left(-2r\pi^2|Y^{t}t|^2\right)\right]$ and then by Hölder the claim (see also Lemma \ref{lemmanorm}).

One can compare the $p$-moments between sums of independent Gaussian mixtures even for $p>2$ directly due to their specific structure. This was the content of \cite[Theorem 3]{ENT} in one dimension and can be directly extended to higher dimensions. We include the short proof for completeness.

\begin{lemma} \label{lemmaENT}
    Suppose that $X_1,\ldots,X_n$ are $n$ independent Gaussian mixtures vectors. Then the function  (defined on $[0,+\infty)^n$)
    \[
    \mathbb{E}\left|\sqrt{a_1}X_1+\cdots+\sqrt{a_n}X_n\right|^p
    \]
    is Schur convex for $p\in (-d,0)\cup[2,+\infty)$ and Schur concave for $p\in (0,2]$, assuming that $\mathbb{E}|X_i|^p <	 +\infty $.
\end{lemma}
\begin{proof}
Since the $X_i$'s are i.i.d. Gaussian mixtures then $X_i=_dY_iZ_i$, with $Y_i,Z_i$ independent and $Z_i$ standard Gaussian random vectors and thus $\sqrt{a_1}X_1+\cdots+\sqrt{a_n}X_n =_d  \sqrt{a_1Y_1Y_1^t+\cdots+a_nY_nY_n^t}Z$, with $Z$ a standard Gaussian independent of the $Y_i$'s. So we write
    \begin{align*}
    \mathbb{E}\left|\sqrt{a_1}X_1+\cdots+\sqrt{a_n}X_n\right|^p_2 
    &    
        =  \mathbb{E}_Y \mathbb{E}_Z \left\langle \sqrt{a_1Y_1Y_1^t+\cdots+a_nY_nY_n^t}Z,\sqrt{a_1Y_1Y_1^t+\cdots+a_nY_nY_n^t}Z  \right\rangle^{p/2} \\
 & 
         =  \mathbb{E}_Y \mathbb{E}_Z \left\langle (a_1Y_1Y_1^t+\cdots+a_nY_nY_n^t)Z,Z  \right\rangle^{p/2} = \mathbb{E}_Y \mathbb{E}_Z \left( \sum_{i=1}^n a_i\|Y_iZ\|_2^2 \right)^{p/2}.
\end{align*}
Notice that the random variables $\|Y_iZ\|_2^2$ are exchangable, that is the distribution of the vector 

$\left(\|Y_{\sigma(i)}Z\|_2^2\right)_{i=1,\ldots,n}$ is independent of the permutation $\sigma \in S_n$ and that $t^{\frac{p}{2}}$ is convex or concave depending on the value of $p$. Now the result follows from Theorem \ref{exchangable theorem} (see also Section \ref{stablemixtures} below).
\end{proof}

It seems that Theorem \ref{centralthm} alone cannot recover the above Schur monotonicity for the whole range $p\in(-1,+\infty)$, but it can rather give the Schur monotonicity for $p\in (-1,2)$ and the larger class of distributions considered. However, it should be noted that when the moments are finite the Schur monotonicity (\ref{boosted 1 dim}) is stronger than the Schur monotonicity of the moments, while it remains true even when these are infinite. In this sense Corollary \ref{momentcorollary} generalizes and complements \cite[Theorem 3]{ENT}.

In the next section we show that Theorem \ref{centralthm} gives more information on the volume of block sections of some star bodies and improves upon Lemma \ref{lemmaENT} for some specific norms.

On the other hand, for the case of $p$-pseudo-stable random variables and in general for distributions satisfying the log-concavity assumption on the Fourier transform, such comparison results (as in Corollary \ref{momentcorollary}) seem to be new.

The next result illustrates the link between our main topic and Khinchin inequalities that were mentioned in the introduction.
\begin{corollary} \label{corollary CLT}
    Let $X$ be a random vector on $\mathbb{R}^d$ with finite second moment distributed according to the probability measure $\mu(dx)=\phi(x)dx$, with $\phi$ even and $\hat{\phi}$ integrable and positive. Let $h:\mathbb{R}^d\to [0,+\infty)$ be a continuous and positive definite function. Let $Z$ be a Gaussian random vector of mean zero and $\mathrm{Cov}(X)=\mathrm{Cov}(Z)$ and $X_1,\ldots,X_n$ are i.i.d. copies of $X$, then we have:

    $\bullet$ If $\Hat{\phi}(\sqrt{\cdot}t)$ is log-convex for each $t\in \mathbb{R}^d$
    \begin{equation} \label{eqCLT}\mathbb{E}h(Z)\leq \mathbb{E}h\left( \sqrt{a_1}X_1+\cdots+ \sqrt{a_n}X_n \right) \leq \mathbb{E}h(X), \end{equation}
for all $a\in \Delta^{n-1}$.

$\bullet$ If $\Hat{\phi}(\sqrt{\cdot}t)$ is log-concave for each $t\in \mathbb{R}^d$ \begin{equation} \label{eqCLT log-conc}\mathbb{E}h(Z)\geq \mathbb{E}h\left( \sqrt{a_1}X_1+\cdots+ \sqrt{a_n}X_n \right) \geq \mathbb{E}h(X), \end{equation}
for all $a\in \Delta^{n-1}$.
\end{corollary}
\begin{proof}
    
Let $X$ be a random vector on $\mathbb{R}^d$, that has finite second moment and that $\Hat{\phi}(\sqrt{\cdot}t)$ is log-convex for every $t\in \mathbb{R}^d$ . Let $\{X_i\}$ be a sequence of independent copies of $X$.
 Let also $h$ as in the statement, that in particular satisfy the conditions of Theorem \ref{centralthm}. Then, for all $a\in \Delta^{n-1}$ applying the Schur convexity result for $\left(\frac{1}{\sqrt{n}},\ldots,\frac{1}{\sqrt{n}},0\right)\prec a \prec (1,0,\ldots,0)$, we have the following chain of inequalities
\begin{equation} \label{applying CLT}
\mathbb{E} h\left(\frac{X_1+\cdots+X_n}{\sqrt{n}}\right)\leq \mathbb{E}h\left( \sqrt{a_1}X_1+\cdots+ \sqrt{a_n}X_n \right)  \leq \mathbb{E}h(X).
\end{equation}
So, by the central limit theorem and our assumptions on $h$ we get the result. Similarly for the log-concave case.

\end{proof}
\begin{remark} \label{remark Khinchine}

In the framework of moment comparison such arguments can (and do) lead to sharp $(2,p)$ Khinchin inequalities. We refer to Corollary 25 of \cite{ENT} for the argument in the one dimensional case. Even in the case of random vectors that we have here, like in Lemma \ref{lemmaENT} and  Theorem \ref{momentsweak} below, at least when the norm is given by an inner product this argument can give Khinchine inequalities. See Corollary \ref{khinchin} below. Applying now the above comparison, passing first through Corollary \ref{momentcorollary}, we get the following Khinchin inequalities.
    
\end{remark}

Applying now Corollary \ref{corollary CLT} to the functions $h_\lambda(t)=\exp(-\lambda|t|^p)$, for $p\in (0,2]$, or directly adapting the above argument to $h(t)=|t|^p$, for $p\in (-1,2]$ via Corollary \ref{momentcorollary} we get:

\begin{corollary} \label{corollary khinchine log-concave}
Let $X_1,\ldots,X_n$ be symmetric i.i.d. random variables having density $\phi$ with integrable and positive Fourier transform $\hat{\phi}$. Let $p\in (-1,2]$ and assume all the second and $p$ moments exist and $Z$ is a Gaussian random variable with $\mathrm{Var}(X_1)=\mathrm{Var}(Z)$. Denote $X=(X_1,\ldots,X_n)$. Then for each $\theta\in S^{n-1}$ it holds that
\[
\bullet \ \ \ \frac{\|Z\|_p}{\|X\|_2} \|\theta_1X_1+\cdots+\theta_nX_n\|_2
\leq \|\theta_1X_1+\cdots+\theta_nX_n\|_p
\leq \frac{\|X\|_p}{\|X\|_2}\|\theta_1X_1+\cdots+\theta_nX_n\|_2,
\]
if $\Hat{\phi}(\sqrt{t})$ is log-concave on $(0,+\infty)$, and
\[
\bullet \ \ \ \frac{\|Z\|_p}{\|X\|_2} \|\theta_1X_1+\cdots+\theta_nX_n\|_2
\geq \|\theta_1X_1+\cdots+\theta_nX_n\|_p
\geq \frac{\|X\|_p}{\|X\|_2}\|\theta_1X_1+\cdots+\theta_nX_n\|_2,
\]
if $\Hat{\phi}(\sqrt{t})$ is log-convex on $(0,+\infty)$.
\end{corollary}

These Khinchin inequalities are sharp and seem to be new in the log-concave case and in the log-convex case for random variables that are not Gaussian mixtures. Otherwise, for Gaussian mixtures and the above range of exponents they recover \cite[Corollary 25]{ENT}, where the argument is taken from. For the higher dimensional version see Corollary \ref{khinchin} below.
\par \vspace{0.5em}

$\bullet$ \textbf{Case $q<2$.} In this case the class of the Gaussian mixtures stop being a "universal" example of log-convexity (or log-concavity) of the function $\hat{\phi}((\cdot)^{\frac{1}{q}}t):[0,+\infty)\to (0,+\infty)$ where as usual by $\phi$ we denote the density and $t\in \mathbb{R}^d$.

In the high dimensional case if we consider a multivariate symmetric $p$-stable random vector $X$ with spectral measure $\sigma_0$ (see Section \ref{section prelims}), we see that, for each $t\in \mathbb{R}^d$ and $a>0$ we have
    \[
    \log\hat{\phi}(a^{\frac{1}{q}}t)= -a^{\frac{p}{q}} \int_{S^{d-1}}|\langle t,\theta
 \rangle|^p\sigma_0(d\theta).
    \]
This shows that depending on the ratio $\frac{p}{q}$, the measure $\mu$ with such density will satisfy the condition of either the log-convexity or the log-concavity of $\hat{\phi}((\cdot)^{\frac{1}{q}}t)$.

In Section \ref{stablemixtures} we comment on a general class of such functions corresponding in the log-convex scenario in the spirit of Gaussian mixtures.

Regarding the pseudo-stable random variables, they still form an example for the log-concave case.

Finally, we would like to mention in passing that for some random variables with infinite variance a version of Corollary \ref{corollary CLT} would be possible, and in that case the limiting distribution would be a $q$-stable distribution, using the generalized Central Limit Theorem.

\par\vspace{1em}
$\bullet$ \textbf{Case $q=1$.} In this particular case, the condition one wants is just the log-concavity or log-convexity of the Fourier transform of the density.
The $p$-stable laws along with the $p$-pseudo-stable random variables are all examples for the log-concave case now. Another example, is the measure with density $\left(\frac{sint}{t}\right)^2$, which also is not a Gaussian mixture.

An interesting feature of this class is that once a density has log-convex Fourier transform, then the associated probability measure will not obey the (weak) law of Large numbers, by applying an argument similar to the case $q=2$ for the central limit theorem.

Indeed, suppose that we are given a probability measure with even density $\phi$, such that its Fourier transform is integrable and log-convex. Applying Theorem \ref{centralthm} for any suitable function $h$, and in particular from the fact that we have the Schur convexity of the function $H(a)=\mathbb{E}h(a_1X_1+\cdots+a_nX_n)$ we get
\[
    \mathbb{E}h\left(  \frac{X_1+\cdots+X_n}{n} \right) \leq \mathbb{E}h(X) .   \]
This shows that for any such probability measure the (weak) law of large numbers cannot hold, since we would have $h(0)\leq \mathbb{E}h(X)$ for any continuous and positive definite $h$ which leads to contradiction, by considering e.g. $h_a(x)=\exp(-ax^2)$ and taking the limit $a\to +\infty$, unless $X$ is almost surely equal to zero.

\par \vspace{1em}

$\bullet$ \textbf{Case $q>2$.} The class of Gaussian mixtures is always an example for the log-convex case, which can be seen using Hölder's inequality. Moreover, once $X$ is a random variable the form $YZ_p$, with $Z_p$ being $p$-pseudo-stable and $Y$ being positive 
random variable independent of $Z_p$, then if $q\geq p$ it is an example for the log-convex scenario. Indeed, we have that
\[
\mathbb{E}\exp(-2\pi it^{1/q}X)=\mathbb{E}_Y \exp\left(-c_1t^\frac{p}{q}Y^p-c_2t^\frac{2}{q}Y^2\right),
\]
which is seen to be log-convex using Hölder.

\section{Higher dimensional applications} \label{Section High dim}

\subsection{Representation of the uniform probability measure on $B_p^n(K)$} As Schechtman and Zinn \cite{SZ} and Rachev and Ruschendorf \cite{RR} did for the cone measure of the ball of $l_p^n$, the following can be proven for the cone measure of $B_p^n(K)$. We provide the proof for completeness: \par 

\begin{proposition} \label{propindep} Let $p>0$ and $X_1,\ldots,X_n$ be i.i.d. random vectors on $\mathbb{R}^d$ with density  $\frac{1}{|K|\Gamma(1+\frac{d}{p})}\exp(-\|x\|_K^p)$.\par
If $X=(X_1,\ldots,X_n)$ then \par
\textbf{1.} $\frac{X}{\|X\|_{l_p^n(K)}}$ induces the measure $\mu_{B_p^n(K)}$ on $S_p^n(K)$.\par
\textbf{2.} $\|X\|^p_{l_p^n(K)}$ follows the Gamma$\left(\frac{nd}{p},1\right)$ distribution and is independent of $\frac{X}{\|X\|_{l_p^n(K)}}$.\end{proposition}

\begin{proof} Let $A$ and $B$ be Borel subsets of $S_p^n(K)$ and $[0,+\infty)$ respectively, then we have
\[
\mathbb{P}\left(\frac{X}{\|X\|_{l_p^n(K)}}\in A \ \text{and} \ \|X\|_{l_p^n(K)}\in B\right) = \frac{1}{|K|^n\Gamma(1+\frac{d}{p})^n} \int_{\mathbb{R}^{nd}}\mathbbm{1}_{\{\frac{x}{\|x\|_{l_p^n(K)}}\in A\}}(x)\mathbbm{1}_{\{\|x\|_{l_p^n(K)}\in B\}}(x) e^{-\|x\|^p_{l_p^n(K)}} \ dx
\]
\[
=\frac{nd|B_p^n(K)|}{|K|^n\Gamma(1+\frac{d}{p})^n} \int_{A}\int_Br^{nd-1}e^{-r^p}drd\mu_{B_p^n(K)}(\theta) = \mu_{B_p^n(K)}(A)\frac{nd}{\Gamma\left(1+\frac{nd}{p}\right)}\int_Br^{nd-1}e^{-r^p}dr,
\]
where we used the polar integration formula (\ref{eq: polar integration formula}) for the body $B_p^n(K)$ and the fact that $|B_p^n(K)|=\frac{|K|^n\Gamma(1+\frac{d}{p})^n}{\Gamma(1+\frac{nd}{p})}$ from (\ref{volume of }). This proves both claims.
\end{proof}

Next, proceeding as Barthe, Guédon, Mendelson and Naor did in \cite{BGMN}, we obtain a representation of the uniform probability measure on $B_p^n(K).$ 

\begin{proposition} \label{representationprop}
Let $n$ i.i.d. random vectors $X_1,\ldots,X_n$ independent and identically distributed as above and $Z$ a random variable independent again of the rest, with density $e^{-t}\mathbbm{1}_{(0,\infty)}(t)$. Then, the random vector \[\frac{(X_1,\ldots,X_n)}{ \left( \|X\|_{l_p^n(K)}^p+Z \right)^{\frac{1}{p}}}\] generates the uniform probability measure on $B_p^n(K)$.\end{proposition} 

\begin{proof}
Let $f\in L_1(\mathbb{R}^{nd})$. For shortness we will denote $\|\cdot\|_{l_p^n(K)}$ by $\|\cdot\|_p$. We write 
\[
\mathbb{E}f\left(\frac{X}{(\|X\|_{l_p^n(K)}^p+Z)^\frac{1}{p}}\right)= \int_0^\infty e^{-t}\mathbb{E}f\left(\frac{X}{(\|X\|_{p}^p+t)^\frac{1}{p}}\right)dt
\]

\[
=\int_0^\infty \int_0^\infty e^{-t}e^{-u}u^{\frac{nd}{p}-1}\frac{1}{\Gamma(\frac{nd}{p})} \mathbb{E}f\left( \left(\frac{u}{u+t}\right)^\frac{1}{p} \frac{X}{\|X\|_p}\right)  dudt,
\]
where we have used the independence property stated in Proposition \ref{propindep}. Setting $\frac{u}{u+t}=r^p$, so $u = \frac{t}{1-r^p}-t$ and $du = \frac{pt}{(1-r^p)^2}r^{p-1}dr$, and we obtain that the above expectation equals
\[
\int_0^\infty \int_0^1 e^{\frac{t}{1-r^p}}t^{\frac{nd}{p}-1}\frac{r^{nd-p}}{(1-r^p)^{\frac{nd}{p}-1}} \frac{1}{\Gamma(\frac{nd}{p})}\mathbb{E}f\left(r\frac{X}{\|X\|_p}\right)\frac{pt}{(1-r^p)^2}r^{p-1}drdt
\]
\[
=\frac{nd}{\Gamma(\frac{nd}{p}+1)}\int_0^1 \frac{r^{nd-1}}{(1-r^p)^{\frac{nd}{p}+1}}\mathbb{E}f\left(r\frac{X}{\|X\|_p}\right) \int_0^\infty e^{\frac{t}{1-r^p}}t^{\frac{nd}{p}} dt dr = nd \int_0^1 r^{nd-1}\mathbb{E}f\left(r\frac{X}{\|X\|_p}\right)dr \]
\[=nd \int_0^1 r^{nd-1}\mathbb{E}f\left(r\frac{X}{\|X\|_p}\right)dr = \frac{1}{|B_p^n(K)|}\int_{B_p^n(K)}f(x) \ dx
\]

\noindent where we used again the integration formula with $L=B_p^n(K)$ (and thus $N=nd$) and the fact that the vector $X/\|X\|_{l_p^n(K)}$ induces the cone measure of $B_{p}^n(K)$.\end{proof}

\begin{remark} As in \cite{BGMN}, one can consider, instead of the exponential density as the law of the random variable $Z$, other densities and produce other measures as a result. \end{remark} 

\subsection{The subspaces of $L_p,$ $ p\in (0,2]$ - Sections}
 Here the star body $K$ that is considered will be such that $(\mathbb{R}^d,\|\cdot\|_K)$ embeds isometrically to $L_p$ for some $p\in (0,2]$. Recently, Eskenazis in \cite{eskenazis} considered these type of bodies and proved the following result, which we state in a way that it fits the notation so far: 

The function \[\Delta^{n-1}\ni (a_1,\ldots,a_n) \mapsto |B_p^n(K)\cap H_{\sqrt{a}}|,\] is Schur convex.

Using the same general strategy and some crucial steps from \cite{eskenazis}, we will give a proof of the following slight strengthening by employing Proposition \ref{sectionlogsth}: 

\begin{theorem} \label{eskethm}
    Let $p\in (0,2]$ and suppose that $(\mathbb{R}^d,\|\cdot\|_K)$ embeds isometrically to $L_p$. Then the function
   \[ (a_1,\ldots,a_n) \mapsto |B_p^n(K)\cap H_{\sqrt{a}}|, \]
   is log-convex in $\Delta^{n-1}$
\end{theorem}

To start with, in \cite{MP} the sections of $p$-balls, were restated to product measures. The key point is the following. Let $N:\mathbb{R}^d\to\mathbb{R}_+$ be a continuous homogeneous function that is zero only at $0\in \mathbb{R}^d$, then for the star body $B=\{x: N(x)\leq1\}$, any  $p>0$ and any $k$-dimensional subspace $E$ of $\mathbb{R}^d$, 

\begin{equation}\label{sectionshigh}
\Gamma\left(1+\frac{k}{p}\right)|E\cap B| = \int_E e^{-N(x)^p}\mathcal{H}^k(dx)= \lim_{\varepsilon\to0} \frac{1}{\varepsilon^{n-k}}\int_{E(\varepsilon)} e^{-N(x)^p}dx
\end{equation}
where if $E=\{u_1,\ldots,u_{n-k}\}^\perp$, $E(\varepsilon):= \{x\in \mathbb{R}^d : |\langle x,u_i\rangle|\leq \frac{\varepsilon}{2}, \text{  for  } i=1,\ldots,n-k\}$.

Choosing $N(x)= \|x\|_{l_p^n(K)}$, we obtain the following expression for block sections of $B_p^n(K)$
\begin{gather} \label{sections of cb}
\Gamma\left(1+\frac{nd-d}{p}\right)|B_p^n(K)\cap\{x=(x_1,\ldots,x_n)\in \mathbb{R}^{nd} \;\; : \sum_{k=1}^n\theta_kx_k=0\} | = \\ \int_{\sum \theta_kx_k=0} \exp(-\|x\|_{l_p^n(K)}^p)\mathcal{H}^{nd-d}(dx)= S_{\mu^n}(\theta),
\end{gather}
for $\mu(dx)=\phi(x)dx$ and $\phi(x)=\exp(-\|x\|_K^p)$. In particular, if we have that $\hat{\phi}(t\sqrt{\cdot})$, is log-convex  for each $t\in \mathbb{R}^d$, applying Proposition \ref{sectionlogsth} proves Theorem \ref{eskethm}. From Section \ref{section prelims} what we know for sure is that this function is positive definite. Putting the body in Lewis' position we have the following: 

\begin{lemma} \label{lemmanorm}
Let $p\in (0,2]$ and suppose that $(\mathbb{R}^d,\|\cdot\|_K)$ embeds isometrically to $L_p$ and $K$ is in Lewis' position. If $\phi=\exp(-\|\cdot\|_K^p)$, then for each $t\in \mathbb{R}^d$, the function $(0,+\infty)\ni s\mapsto\hat{\phi}(t\sqrt{s})$ is log-convex.\end{lemma}
\begin{proof}
From Section \ref{section prelims} we know that when $K$ is in Lewis' position there exists an isotropic measure $\mu$ in $S^{d-1}$ so that the following holds:
\[
e^{-\|x\|_K^p}=\exp\left( -\int_{S^{d-1}} |\langle\theta,x \rangle|^p d\mu(\theta) \right)
\]

So denote $\|\cdot\|:=\|\cdot\|_K$ and let $\mu$ be a Borel isotropic measure on the $S^{d-1}$. From \cite{bartheapprox}, for each $k\in \mathbb{N}$, there exist $M_k\in \mathbb{N}$, $c_{j,k}>0 $ and $u_{j,k}\in S^{d-1}$ such that the discrete measure 
\[
\mu_k= \sum_{j=1}^{M_k} c_{j,k}\delta_{u_{j,k}}
\]
is isotropic and converges weakly to $\mu$. 

In particular, if we denote $\|x\|_k^p=\int_{S^{d-1}} |\langle\theta,x \rangle|^p d\mu_k(\theta)$ and $\phi_k=\exp(-\|\cdot\|_k^p)$, we have that for each $y \in \mathbb{R}^d$, $\hat{\phi}_k(y)\to \hat{\phi}(y)$. Hence, in order to prove the desired property for $\hat{\phi}$ it suffices to prove it for $\hat{\phi}_k$.

From now on we fix $M\in \mathbb{N}$, vectors $u_1,\ldots,u_M$ on $S^{d-1}$ and $c_1,\ldots,c_M>0$ so that the discrete measure 
$ \sum_{j=1}^{M} c_{j}\delta_{u_{j}}$
is isotropic, and let $p\in (0,2]$. Then we define a quasi-norm on $\mathbb{R}^d$ by $\|x\|^p=\sum_{j=1}^M c_j|\langle x,u_j\rangle|^p$ and $\phi(x):=\exp(-\|x\|^p)$.

By Bernstein's theorem and since $p/2\in (0,1]$, for each $j\in[M]$, there exists a positive probability measure $\nu_j$ on $(0,+\infty)$, such that $\exp(-c_j|t_j|^{p/2})=\int_0^\infty \exp(-a_jt_j)d\nu_j(a_j).$ \par
Applying the above to $t_j=\langle x,u_j\rangle^2$ and taking the product over all $j\in [M]$, we arrive at the following convenient formula, where we denote $\nu = \nu_1\otimes\cdots\otimes\nu_M$ and $a=(a_1,...,a_M)$,
\begin{equation} \label{eq:embGM}
    \phi(x)=e^{-\|x\|^p}=\int_{(0,+\infty)^M}e^{-\sum_{j=1}^Ma_j\langle x,u_j\rangle^2}d\nu(a).
\end{equation}

Taking the Fourier transform yields 
\[
\hat{\phi}(y)=\int_{(0,+\infty)^M}\int_{\mathbb{R}^n} e^{-\sum_{j=1}^Ma_j\langle x,u_j\rangle^2-2\pi i\langle x,y\rangle}dxd\nu(a)
\]
If now, for each $a\in (0,+\infty)^M$, we denote by $A_a$ the symmetric and positive definite matrix $A_a:= \sum_{j\in[M]}a_ju_j\otimes u_j$ (which is of full rank because of isotropicity), we see that the inside integral is the Fourier transform of (a multiple of) a Gaussian density and so we can rewrite the Fourier transform of $\phi$ as 
\[
\hat{\phi}(y)=\int_{(0,\infty)^M} \frac{\pi^\frac{n}{2}}{\sqrt{\det(A_a)}} e^{- \pi^2 \langle A^{-1}_ay,y\rangle}d\nu(a).
\]
In particular using the above expression we see that
\[
\hat{\phi}(t\sqrt{s})=\int_{(0,+\infty)^M}\frac{\pi^\frac{n}{2}}{\sqrt{\det(A_a)}} \left(e^{- \pi^2 \langle A^{-1}_at,t\rangle}\right)^sd\nu(a)
\]
and thus applying Holder's inequality we get that it is indeed log-convex in $s$.

\end{proof}

\begin{proof}[Proof of Theorem \ref{eskethm}]

When $K$ is in Lewis' position then by Lemma \ref{lemmanorm} the function $\phi(x):=\exp(-\|x\|_K^p)$  satisfies the log-convexity assumption of Proposition \ref{sectionlogsth}. From (\ref{sections of cb}), this yields that the function
\[
(a_1,\ldots,a_n)\mapsto |B_p^n(K)\cap H_{\sqrt{a}}|
\]
is log-convex on $\Delta^{n-1}$, and proves the theorem with the extra assumption that $K$ is in Lewis' position.

For general $K$, by \cite[Lemma 2.1]{eskenazis}, we have that for any $T\in GL_n$ and $a\in \Delta^{n-1}$,
\[
|B_p^n(K)\cap H_{\sqrt{a}}|=\det(T)^{1-n}|B_p^n(K)\cap H_{\sqrt{a}}|.
\]
Thus applying the suitable invertible linear $T$ to bring the body in Lewis' position proves the result for general $K$.
\end{proof}

Because of the symmetry (that is invariance under permutations) of the section function with respect to the vector $a$, one gets (and improves upon) the Schur monotonicty and the extremals as well that were proven in \cite{eskenazis}.\par 
\begin{remark} \label{remark laplace is section}
   In the original proof the idea was to prove first a Schur monotonicity of the Laplace transform of Gaussian random vectors (\cite[Theorem 2.3]{eskenazis}) in the spirit of \cite{BGMN} and \cite{ENT}, that also allowed to get estimates for the mean width of projections.  In our terminology Theorem \ref{eskethm} followed from the following ``more general" statement:

   If $G_a$ is a standard Gaussian random vector on $H_{\sqrt{a}}$, then for each $\lambda>0$ the function
   \begin{equation} \label{laplace comp eske}
       \Delta^{n-1} \ni (a_1,\ldots,a_n) \mapsto \mathbb{E}\exp\left( -\lambda \|G_a\|_{l_p^n(K)\cap H_{\sqrt{a}} }\right)
   \end{equation}
    is Schur convex.
    
We would like to point out here that an argument as above could give this statement as well, which shows that in fact the above Schur monotonicity, i.e. of the function (\ref{laplace comp eske}), is as general as that of Theorem \ref{eskethm}.

Indeed, by \cite[Lemma 10]{BGMN} and explicitly \cite[Lemma 2.5]{eskenazis} we have that 
\begin{equation} \label{laplace is prod}
   \mathbb{E}\exp\left( -\lambda \|G_a\|_{l_p^n(K)\cap H_{\sqrt{a}} }\right)= \lim_{\varepsilon\to0^+} \frac{c}{\varepsilon^{d}} \mu_p^n(H_{\sqrt{a}}(\varepsilon)),
\end{equation}
for some constant $c$, where $\mu_p(dx):= \exp\left( -\lambda\|x\|_K^p-\frac{1}{2}|x|^2 \right)dx$. Combining this now with (\ref{sectionshigh}), it is indeed the block section of the product measure $\mu_p^n$. Now it is not hard to see, arguing as in Lemma \ref{lemmanorm} (see also in \cite[Lemma 2.6]{eskenazis}), that its density satisfies the properties we want -that is the Fourier transform of the density satisfies the log-convexity assumption-, thus getting the above statement even with log-convexity, instead of Schur convexity.
    
   Even though it seems that the present approach simplifies the argument, providing a unified approach, it should be clear that the analysis of \cite{eskenazis}, passing to Lewis' position and approximating with discrete isotropic measures, was very important even now.
\end{remark}

\subsection{The subspaces of $L_p,$ $ p\in (0,2]$ - Moment comparison} Returning now to the comparison of moments, by Lemma \ref{lemmaENT} we have the Schur convexity in $a=(a_1,\ldots,a_n)$ of $\mathbb{E}\|\sqrt{a_1}X_1+\cdots+\sqrt{a_n}X_n\|^p_2$, whenever $X_1,\ldots,X_n$ are independent random (vector) Gaussian mixtures. \par
In the setting of subspaces of $L_p$, we know that if a quasi-normed space $(\mathbb{R}^d,\|\cdot\|)$ embeds isometrically into $L_p$ for $p\in (0,2]$ we have that for each $\lambda>0$, the continuous function 
\[
h(x)=e^{-\lambda\|x\|^p}
\]
is positive definite. Thus using Theorem \ref{centralthm} we get the following (the proof is omitted as it is identical to that of Corollary \ref{momentcorollary}):

\begin{corollary} \label{highmom}
    Let $q>0$ and $X_1,\ldots,X_n$ be independent identically distributed random vectors according to a probability measure $\mu(dx)=\psi(x)dx$ with $\psi$ even with positive integrable Fourier transform and such that for each $t\in\mathbb{R}^d$, $[0,+\infty)\ni s\mapsto\Hat{\psi}(s^\frac{1}{q}t)$ is log-convex for $q>0$. Let also $(\mathbb{R}^d,\|\cdot\|)$ be a quasi-normed space that embeds isometrically into $L_p$ for $p\in (0,2]$. Then the function
    
    \begin{equation} \label{boosted high dim}
(0,+\infty)^n \ni a \mapsto \mathbb{E} e^{-\lambda\|\sqrt[q]{a_1}X_1+\ldots+ \sqrt[q]{a_1}X_n\|^p}
\end{equation}
    is log-convex.
        
    Moreover, if $\mathbb{E} \|X_1\|^p<+\infty$ , we have that 
    \[
    [0,+\infty)^d\ni(a_1,\ldots,a_d)\mapsto \mathbb{E}\|a_1^{\frac{1}{q}}X_1+\cdots+a_n^{\frac{1}{q}}X_n\|^p
    \]
    is Schur concave.
    In the case where $\Hat{\psi}((\cdot)^\frac{1}{q}t)$ is log-concave then in the above statements we have that the functions are Schur concave and Schur convex respectively.
\end{corollary}

\begin{remark}
    As was noted in the previous section, for $q=2$ the Gaussian mixture vectors satisfy the condition of log-convexity, while for any intermediate value $0<q\leq2$ the multivariate symmetric $p$-stable distributions can give examples for both conditions depending on the ratio $\frac{p}{q}$.
 
Also let us stress once more that the Schur monotonicity of the function (\ref{boosted high dim}) is a formally stronger statement than the Schur monotonicity of the moments, for that range of $p$. In particular, this result allows to get more general comparisons than in Lemma \ref{lemmaENT} for random vectors that satisfy the condition (eg Gaussian mixtures) and that particular range of $p$.
  
 More examples of comparison with different norms and different range of exponents, that include the quasi-norms considered here, will be presented in the next subsection, where we will have to put some extra assumption on the measure.

 The analogous of the third statement of Corollary \ref{momentcorollary} holds here as well, but besides the extra assumptions on the density it is included in Theorem \ref{momentsweak} (since as it is mentioned in Section \ref{section prelims} any finite dimensional subspace of $L_p$, for $p \in (0, 2]$ embeds in $L_{-l}$ for every $ l \in (0,d)$).
\end{remark}

\subsection{Comparison using embeddings in $L_{-l}$}
In order to apply our previous results in this case we should be careful, since we no longer deal with integrable functions and the Fourier transform is only defined in the weak sense as described in Section \ref{section prelims}. So we will already start with a measure satisfying some stronger assumption and take the convolution of it with a Gaussian kernel to make it belong in the Schwartz space and then try to proceed as in Section \ref{Section main} with a limiting argument. Below we restrict ourselves to the case the density satisfies the log-convexity assumption on the Fourier transform, mainly due to the absence of non-trivial (i.e. non-product) examples. The analogous statements for the log-concave case should be clear.

We have the following result, of which a particular case is Theorem \ref{intersectionintro} (see Section \ref{section prelims} for the relation of intersection bodies and subspaces of $L_{-l}$):

\begin{theorem} \label{momentsweak}
    Let $X_1,\ldots,X_n$ be independent random vectors in $\mathbb{R}^d$ distributed according to the probability measure $\mu(dx)=\phi(x)dx$. Suppose that $\phi$ is even and has positive and integrable Fourier transform, such that for each $t\in \mathbb{R}^d$ the function $ r\mapsto\Hat{\phi}(\sqrt{r}t)$ is log-convex on $[0,+\infty)$. Assume moreover that for each $N>0$ $\mathbb{E}|X_1|^N<+\infty$. Let also $l \in (0,d)$ and consider a quasi-normed space $(\mathbb{R}^d,\|\cdot\|)$ that embeds in $L_{-l}$ and also $\mathbb{E} \|X_1\|^{-l}<+\infty$.
    
    Then the function
\[
 (a_1,\ldots,a_n)\mapsto \mathbb{E} \|\sqrt{a_1}X_1+\cdots+\sqrt{a_n}X_n\|^{-l}
\]
is log-convex on $[0,+\infty)^n$.
\end{theorem}

\begin{remark}
An example of such $\phi$ is given by $\phi(x)=\exp(-\|x\|_\ast^p)$ when $(\mathbb{R}^d,\|\cdot\|_\ast)$ embeds in $L_p$ for some $p\in (0,2]$.
\end{remark}

\begin{proof}

For $\varepsilon>0$ consider the heat kernel $g_\varepsilon(x)=\frac{1}{(\sqrt{2\pi \varepsilon})^d}\exp(-\frac{1}{\varepsilon}|x|^2)$, which belongs to $\mathcal{S}(\mathbb{R}^d)$, and the measure $\mu_\varepsilon$ with density $\phi_\varepsilon:=\phi\ast g_\varepsilon$. The following hold:
\begin{enumerate}[label=(\roman*)]
\item \label{convol belongs in Schwartz} By the assumption on the moments of $\mu$ we get that $\phi_\varepsilon\in \mathcal{S}(\mathbb{R}^d)$. Indeed, for $N\in \mathbb{N}$ and $\alpha=(\alpha_1,...,\alpha_d)\in \mathbb{Z}^d_+$, 
\[
(1+|x|)^N|\partial^\alpha \phi_\varepsilon(x)|\leq  (1+|x|)^N\int_{\mathbb{R}^d}|\partial^\alpha g_\varepsilon(x-y)|\phi(y)dy\leq \int_{\mathbb{R}^d}C_{\alpha,N,g_\varepsilon} \frac{(1+|x|)^N}{(1+|x-y|)^N}\phi(y)dy
\]
\[
\leq C_{\alpha,N,g_\varepsilon} \int_{\mathbb{R}^d} (1+|y|)^N\phi(y)dy <\infty
\]
while the differentiablity is obvious.
\item  Now one can show that $f_\varepsilon(t):=S_{\mu_\varepsilon^n}(\beta,t)=\int_{\langle\beta,x \rangle=t}\phi_\varepsilon(x_1)\cdots\phi_\varepsilon(x_d)\mathcal{H}^{nd-d}(dx)\in \mathcal{S}(\mathbb{R}^d)$, for each $\beta\in [0,+\infty)^n$. Indeed, according to (\ref{fourierofsec}), its Fourier transform is
\[
\Hat{f}_\varepsilon(s)=|\beta|^d \prod_{k=1}^n\hat{\phi}_\varepsilon(\beta_ks)
\]
and since the product of Schwartz functions belongs again in $\mathcal{S}(\mathbb{R}^d)$, we get that $\hat{f}_\varepsilon\in \mathcal{S}(\mathbb{R}^d)$, hence $f_\varepsilon\in \mathcal{S}(\mathbb{R}^d)$.

\item The measure $\mu_\varepsilon$ also has the property that for each $t\in \mathbb{R}^d$ the function $[0,+\infty)\ni r\mapsto\Hat{\phi}_\varepsilon(\sqrt{r}t)$ is log-convex. This is true since $\hat{\phi}_\varepsilon=\hat{\phi}\cdot\hat{g}_\varepsilon$ and both $\phi$ and $g_\varepsilon$ satisfy the log-convexity condition.
\end{enumerate}
Let $h(x)= \frac{1}{\|x\|^{l}}$. By the hypothesis, $h$ gives rise to a positive definite distribution. Applying the representation of Definition \ref{definition embed negative}, plugging in the Schwartz function $f_\varepsilon(t):=S_{\mu_\varepsilon^n}(\sqrt{a},t)$, we proceed below as in previous calculations (recall that $\sqrt{a}=(\sqrt{a_1},\ldots,\sqrt{a_n})$)
\[
\int_{\mathbb{R}^{nd}}h(\langle\sqrt{a},x \rangle)\mu_\varepsilon^n(dx)= \int_{\mathbb{R}^d}h(t)f_\varepsilon(t)dt = \int_0^\infty \int_{S^{d-1}}r^{l-1}\Hat{f}_\varepsilon(r\theta)d\sigma_l(\theta)dr
\]
\[
= \int_0^\infty \int_{S^{d-1}}r^{l-1}\prod_{k=1}^n\hat{\phi}_\varepsilon(\sqrt{a}_kr\theta)d\sigma_l(\theta)dr.
\]
So by Hölder's inequality, and if we denote by $X_i^\varepsilon$ the i.i.d. vectors distributed according to $\mu_\varepsilon$, we get that the function
\[
[0,+\infty)^n\ni (a_1,\ldots,a_n)\mapsto \mathbb{E} \|\sqrt{a_1}X_1^\varepsilon+\cdots+\sqrt{a_n}X_n^\varepsilon\|^{-l}
\]
is log-convex. Letting $\varepsilon\to 0^+$, and using e.g. \ref{convol belongs in Schwartz} and dominated convergence, yields the result for the measure $\mu$, assuming that $\mathbb{E}_\mu \|X\|^{-l}<+\infty$.

\end{proof}

Expanding now on Remark \ref{remark Khinchine}, if we have an ellipsoid $\mathcal{E}$, using the symmetry of the random vectors, for $a\in \Delta^{n-1}$ we see that
\[
\mathbb{E}\|\sqrt{a_1}X_1+\cdots+\sqrt{a_n}X_n\|_{\mathcal{E}}^2= \mathbb{E}\|X\|_{\mathcal{E}}^2.
\]
Given a star body $K\subset \mathbb{R}^d$, we will denote \[
\left(\mathbb{E}\|X\|_K^p\right)^\frac{1}{p}=\|X\|_{K,p}.
\]
Arguing as in Corollary \ref{corollary CLT}, with the help of Corollary \ref{highmom} and Theorem \ref{momentsweak}, we get the following sharp vector Khinchin inequality between $\|\cdot\|_{K,p}$ and $\|\cdot\|_{\mathcal{E},2}$ for certain choices of $(K,p)$:
   \begin{corollary} \label{khinchin}
        Let $X_1,\ldots,X_n$ be independent random vectors in $\mathbb{R}^d$ distributed according to the probability measure $\mu(dx)=\phi(x)dx$ as in Theorem \ref{momentsweak}. Let also $l \in (0,d)$ and consider a normed space $(\mathbb{R}^d,\|\cdot\|_K)$ that embeds in $L_{-l}$ and $Z$ a standard Gaussian random vector with $Cov(Z)=Cov(X)$, then for any $a\in \Delta^{n-1}$ and any ellipsoid $\mathcal{E}\subset \mathbb{R}^d$ we have
\[
\frac{\|X\|_{K,-l}}{\|X\|_{\mathcal{E},2}}\|\sqrt{a_1}X_1+\cdots+\sqrt{a_n}X_n\|_{\mathcal{E},2}\leq \|\sqrt{a_1}X_1+\cdots+\sqrt{a_n}X_n\|_{K,-l}\leq \frac{\|Z\|_{K,-l}}{\|X\|_{\mathcal{E},2}}\|\sqrt{a_1}X_1+\cdots+\sqrt{a_n}X_n\|_{\mathcal{E},2}.
\]
        Moreover, without the extra integrability assumptions of Theorem \ref{momentsweak}, if we assume that $(\mathbb{R}^d,\|\cdot\|_K)$ embeds in $L_{p}$, for $p\in (0,2)$, with the same notation for $Z$ we have that:
\[
\frac{\|X\|_{K,p}}{\|X\|_{\mathcal{E},2}}\|\sqrt{a_1}X_1+\cdots+\sqrt{a_n}X_n\|_{\mathcal{E},2}\leq \|\sqrt{a_1}X_1+\cdots+\sqrt{a_n}X_n\|_{K,p}\leq \frac{\|Z\|_{K,p}}{\|X\|_{\mathcal{E},2}}\|\sqrt{a_1}X_1+\cdots+\sqrt{a_n}X_n\|_{\mathcal{E},2}.
\]
        
   \end{corollary}

  \begin{remark}   Arguing as in the proof of Theorem \ref{momentsweak} one can see that, for the class of measures described above, similar log-convexity or Schur concavity results can hold for distributions $h$ which are positive definite, thus making it possible to further generalize (in the sense of the function $h$) in some cases Theorem \ref{centralthm} using a generalization Bochner's Theorem proved by Schwartz, which states that every positive definite distribution $h:\mathbb{R}^d\to \mathbb{R}$ (or even $T\in \mathcal{S}'$) is the Fourier transform of a tempered measure on $\mathbb{R}^d$.
\end{remark}

Using now the representation of Section \ref{Section High dim}.1 and the comparison of Theorem \ref{momentsweak}, arguing similarly to \cite{BGMN} and \cite{ENT} we have the following comparison of moments for random vectors distributed on the ball $B_p^n(K)$. Below $Y\in \mathbb{R}^{nd}$ and so for $a\in [0,\infty)^n$
\[
\langle\sqrt{a},Y\rangle := \sqrt{a_1}\ Y_1+\cdots+\sqrt{a_n}\ Y_n\in \mathbb{R}^d.
\]

\begin{corollary}\label{corollary moments random vector}
   Let $K\subset \mathbb{R}^d$ be a star body in Lewis' position for $p\in (0,2]$  and let $l \in (0,d)$ and a quasi-normed space $(\mathbb{R}^d,\|\cdot\|)$ that embeds in $L_{-l}$. Then, if $Y$ is a random vector uniformly distributed on $B_p^n(K)$, the function 
   \[
[0,+\infty)^n\ni (a_1,\ldots,a_n)\mapsto \mathbb{E} \|\langle\sqrt{a},Y\rangle\|^{-l}
\]
is log-convex.
\end{corollary}
\begin{proof}
    To see this we re-write the expression $\mathbb{E} \|\langle\sqrt{a},Y\rangle\|^{-l}=\mathbb{E} \|\sqrt{a_1}Y_1+\cdots+\sqrt{a_n}Y_n\|^{-l}$ using the standard argument already from \cite{BGMN} and \cite{ENT}.

According to Proposition \ref{representationprop}, $Y$ has the same law as

\[
\frac{(X_1,\ldots,X_n)}{ \left( \|X\|_{l_p^n(K)}^p+Z \right)^{\frac{1}{p}}},
\]
where $X_1,\ldots,X_n$ are i.i.d. random vectors on $\mathbb{R}^d$ with density $c_{K,n,p}exp(-\|x\|^p_K)$ and $Z$ is a random variable independent of the rest with density $e^{-t}\mathbbm{1}_{(0,+\infty)}(t)$ .

    By the independence of $X/\|X\|_{l_p^n(K)}$ and $\|X\|_{l_p^n(K)}$, from Proposition \ref{propindep}, we may write
    \[
    \mathbb{E} \|\sqrt{a_1}Y_1+\cdots+\sqrt{a_n}Y_n\|^{-l}=\mathbb{E} \left\|\sqrt{a_1}\frac{X_1}{ \left( \|X\|_{l_p^n(K)}^p+Z \right)^{\frac{1}{p}}}+\cdots+\sqrt{a_n}\frac{X_n}{ \left( \|X\|_{l_p^n(K)}^p+Z \right)^{\frac{1}{p}}}\right\|^{-l}=
    \]
    \[
    =\mathbb{E}\frac{\|X\|_{l_p^n(K)}^{-l}}{ \left( \|X\|_{l_p^n(K)}^p+Z \right)^{-\frac{l}{p}}}\mathbb{E}\left\|\sqrt{a_1}\frac{X_1}{ \|X\|_{l_p^n(K)}}+\cdots+\sqrt{a_n}\frac{X_n}{ \|X\|_{l_p^n(K)}}\right\|^{-l}=
    \]
\[
=\mathbb{E}\frac{\|X\|_{l_p^n(K)}^{-l}}{ \left( \|X\|_{l_p^n(K)}^p+Z \right)^{-\frac{l}{p}}} \frac{1}{\mathbb{E}\|X\|_{l_p^n(K)}^{-l}}\mathbb{E}\left\|\sqrt{a_1}X_1+\cdots+\sqrt{a_n}X_n\right\|^{-l}.
\]
    The log-convexity follows from the previous theorem.
\end{proof}

\section{Final Remarks} \label{section remarks}

\subsection{Sections of higher co-dimension} In the same spirit as in Section 2, one can write a formula (which we present hoping that it will be of some use), using the Fourier approach, for subspaces of higher co-dimension, which geometrically corresponds to higher co-dimensional section of convex bodies. We stick to the case of product measures $\mu^n$, where $\mu$ is a (probability say) measure on $\mathbb{R}$ having density $\phi$, for simplicity.

Let $k\geq 2$ and $u_1,\ldots,u_k$ be unit vectors, orthogonal to each other. We set $\mathbf{u}:=(u_1,\ldots,u_k)$ and define $F_\mathbf{u}:\mathbb{R}^n\to \mathbb{R}^k$ via 
\[
x\mapsto \left(\langle u_1,x\rangle,\ldots,\langle u_k,x\rangle  \right)
\]
Then, let $h:\mathbb{R}^k\to \mathbb{R}$ and consider the functional $H:\mathcal{U}\subset\left(S^{n-1}\right)^k\to \mathbb{R}$ defined as

\[
H(\mathbf{u})=\int_{\mathbb{R}^n}h(F_\mathbf{u}(x))\mu^n(dx),
\]
where $\mathcal{U}:= \{ \mathbf{u}=(u_1,\ldots,u_k)\in \left(S^{n-1}\right)^k | \langle u_i,u_j\rangle=\delta_{ij} \}$.

Using the co-area formula we re-write it as 
\[
 H(\mathbf{u})=\int_{\mathbb{R}^k}h(t_1,\ldots,t_k)\int_{F_\mathbf{u}^{-1}(t_1,\ldots,t_k)}\phi(x_1)\cdots\phi(x_n)\mathcal{H}^{n-k}(dx)dt_1\cdots dt_k.
\]
In general, for a measure $\nu(dx)=\psi(x)dx$ on $\mathbb{R}^n$, we may define the auxiliary "$k$ co-dimensional section function" in analogy to (\ref{sectionshigh}) as, 
\begin{equation}
S_{\nu,k}\left((u_1,\ldots,u_k),(t_1,\ldots,t_k)\right):= \int_{F_\mathbf{u}^{-1}(t_1,\ldots,t_k)}\psi(x)\mathcal{H}^{n-k}(dx)
\end{equation}
defined on $\mathcal{U}\times \mathbb{R}^k$. Taking its Fourier transform we see that
\begin{equation}\label{highfourier1}
\widehat{S_{\nu,k}}(s_1,\ldots,s_k)=\int_{\mathbb{R}^n}\psi(x)e^{-2\pi i\sum^k_{l=1}\langle u_l,x\rangle s_l   }dx=\Hat{\psi}(\langle u^1,s\rangle,\ldots,\langle u^n,s\rangle  ).
\end{equation}
Since in our case $\psi(x)=\prod^n_{m=1} \phi(x_m)$ we then get
\begin{equation} \label{highfourier2}
\widehat{S_{\mu^n,k}}(s_1,\ldots,s_k)=\prod_{m=1}^n\Hat{\phi}(\sum_{l=1}^k u_l^ms_l )= \prod_{m=1}^n\Hat{\phi}(\langle u^m,s\rangle ) 
\end{equation}
where, by $u_l^m$ we denote the $m$-th coordinate of the vector $u_l$ and by $u^m=(u_1^m,\ldots,u_k^m)$, the vector of the $m$-th coordinates.

So, we end up with the "convenient" 
 \[
 H(\mathbf{u})
=\int_{\mathbb{R}^k}\prod_{m=1}^n\Hat{\phi}(\langle u^m,s\rangle ) \Hat{h}(-s)ds.
\]

\subsection{$q$-stable mixtures and $q$-mixtures} \label{stablemixtures}
We introduce now a class of random variables, which is a subclass of Gaussian mixtures, and which enables us to prove moment comparison in the sense of Corollary \ref{momentcorollary}, arguing though as in Lemma \ref{lemmaENT}, when the weights $\sqrt{a_i}$ are replaced by $a_i^\frac{1}{q}$. They correspond to these type of weights as Gaussian mixtures do for $\sqrt{a}$. Since the proof of the comparison is of different spirit compared to the previous proofs it is put separately here.

Fix some $q\in(0,2)$. We will say that a random variable $X$ is $q$-stable mixture if $X=_dYZ_q$ for independent random variables $Y$ and $Z_q$, with $Y$ being a positive random variable and $Z_q$ a symmetric $q$-stable random variable. Such random variable is of course a Gaussian mixture as well, since $Z_q$ is in turn a product of two independent random variables (a positive and a Gaussian). Since $q$ is now fixed we denote below $Z_q=Z$.

Notice that the Fourier transform of the density of a stable mixture $X$ equals, if we denote its density by $\phi$,
\[
\hat{\phi}(t)= \mathbb{E}e^{-2\pi itYZ}=\mathbb{E}_Y\mathbb{E}_Ze^{-2\pi itYZ}=\mathbb{E}_Ye^{-c|t|^qY^q}.
\]
In particular, by Hölder, the function $(0,\infty)\ni t\mapsto \hat{\phi}(t^\frac{1}{q})$ is log-convex. Hence, one sees by taking a standard Gaussian that the class of $q$-stable mixtures is a strict subset of the Gaussian mixtures.

Let now $X_1,\ldots,X_n$ be i.i.d. $q$-stable mixtures. In particular suppose $X_i=_dY_iZ_i$, where all the random variables involved are independent, $Y_i$ are i.i.d. and almost surely positive and $Z_i$ are i.i.d. $q$-stable. Let also $a_1,\ldots,a_n\geq 0$. We know by independence the following
\[
a_1^{\frac{1}{q}}Y_1Z_1+\cdots+a_n^\frac{1}{q}Y_nZ_n=_d\left(a_1Y_1^q+\cdots+a_nY_n^q\right)^\frac{1}{q}Z
\]
where $Z$ is a $q$-stable random variable independent of the rest.

So for $p\in (0,q)$ we can write the following
\[
\mathbb{E}|a_1^{\frac{1}{q}}X_1+\cdots+a_n^\frac{1}{q}X_n|^p=\mathbb{E}|\left(a_1Y_1^q+\cdots+a_nY_n^q\right)^\frac{1}{q}Z|^p =\mathbb{E}|Z|^p\mathbb{E}|a_1Y_1^q+\cdots+a_nY_n^q|^\frac{p}{q}.
\]

Now as for Lemma \ref{lemmaENT}, we make use of Theorem \ref{exchangable theorem}. In our case the exchangable random variables are the $Y_i^q$, while considering the function $f_1(t)=t^\frac{p}{q}$ we get that  the function
    \[
    [0,\infty)^n \ni (a_1,\ldots,a_n)\mapsto \mathbb{E}|a_1^{\frac{1}{q}}X_1+\cdots+a_n^\frac{1}{q}X_n|^p
    \]
    is Schur concave for $0<p< q$, and in general one has the following, considering the exchangable random variables $(Y_iZ)^q$ (compare also with Theorem \ref{centralthm}):

\begin{corollary}
    Let $q\in (0,2)$ and $X_1,\ldots,X_n$ be i.i.d. $q$-stable mixtures. Let also $h:\mathbb{R}\to \mathbb{R}$ be an even measurable function and suppose that $\mathbb{E}|h(X_1)|,\mathbb{E}|h\left(X_1+\cdots+X_n\right)|<\infty$.  Then if the function $(0,\infty)\ni t\mapsto h(t^{\frac{1}{q}})$ is convex (resp. concave), it implies that the (symmetric) function
    \[
    [0,\infty)^n \ni (a_1,\ldots,a_n)\mapsto \mathbb{E}h\left(a_1^{\frac{1}{q}}X_1+\cdots+a_n^\frac{1}{q}X_n\right)
    \]
    is Schur convex (resp. Schur concave).
\end{corollary}

Having introduced the $q$-stable mixtures we close by introducing the following (again natural to define) family of random variables, and call them $q$-mixtures, but now for each $q>0$. They complement is some sense they $q$-stable mixtures just defined.

We say that $X$ is a $q$-mixture if $X=_dYZ_q$ where $Y$ and $Z_q$ are independent, $Y$ is positive and $Z_q$ has density $c_qe^{-|t|^q}$.

For $q\in (0,2]$ we see again that this is a (strict) subclass of the Gaussian mixtures, and in that case it seems that the two classes defined here improve upon certain features of the Gaussian mixtures.

For the case of $q$-mixtures the following is now immediate using Bernstein's Theorem and proceeding exactly as in the proof Theorem 2 of \cite{ENT}.
\begin{proposition}
    A symmetric random variable $X$ with density $\phi$ is a $q$-mixture if and only if the function $(0,\infty)\ni t\mapsto \phi(t^\frac{1}{q})$ is completely monotonic.
\end{proposition}
\vspace{0.5em}

\subsection{Variants of the main theorem} We would like to close with some variants of the main Theorem \ref{centralthm} and few remarks that might be useful for applications that we state for simplicity in the "one" dimensional case. The proofs follow along the same lines.

\begin{proposition}\label{generalnonproduct} Let $\mu(dx)=\phi(x)dx$ be a probability measure on $\mathbb{R}^n$, with $\phi$ even and having integrable and positive Fourier transform and $h$ continuous positive definite and assume that $\phi$ is in addition symmetric in its arguments. Then if for each $t\in \mathbb{R}$ the function $ [0,+\infty)^n\ni a\mapsto \Hat{\phi}(ta^\frac{1}{q})$ is Schur monotone then the function $H(a_1,\ldots,a_n)=\int_{\mathbb{R}^{n}}h\left( a_1^\frac{1}{q}x_1+\cdots+a_n^\frac{1}{q}x_n\right)\mu(dx)$ has the same Schur monotonicity. \end{proposition}

And lastly a case that cannot give Schur monotonicity, because of lack of symmetry under permutations (or in probabilistic language non-identically distributed -but still independent-  random variables), but could be of use in other situations.

\begin{proposition} \label{generaldiffmeasures}
    Let $n\in \mathbb{N}$ and for each $1\leq i \leq n$, $\mu_i(dx)=\phi_i(x)dx$ is a probability measure on $\mathbb{R}$. Suppose we are given in addition an integrable function $h:\mathbb{R}\to \mathbb{R}$ and that $\phi_i,h$ are even with positive and integrable Fourier transforms and that all $\Hat{\phi_i}((\cdot)^\frac{1}{q})$ are log-convex. Then the function $H:[0,\infty)^n\to\mathbb{R}$ defined as $H(a_1,\ldots,a_n)=\int_{\mathbb{R}^{n}}h\left( a_1^\frac{1}{q}x_1+\cdots+a_n^\frac{1}{q}x_n\right)\mu(dx)$ is log-convex.

    While the converse holds as well, i.e. if the functional $H$ is log-convex for each such $h$, then $\Hat{\phi_i}((\cdot)^\frac{1}{q})$ is log-convex for each $i=1,\ldots,n$.
\end{proposition}

\subsection*{Acknowledgments} I am very grateful to Franck Barthe for all his help to improve the presentation, to Alexandros Eskenazis for helpful discussions and sharing his insights in related problems and to Piotr Nayar and Cyril Roberto for valuable comments and useful discussions. This work received support from the Graduate School EUR-MINT (State support managed by the National Research Agency for Future Investments program bearing the reference ANR-18-EURE-0023).

\bibliographystyle{alpha}
\newcommand{\etalchar}[1]{$^{#1}$}

\end{document}